\documentclass[11pt]{amsart}

\usepackage[utf8]{inputenc}
\usepackage[T1]{fontenc}
\usepackage{amsmath, amssymb, amsthm}
\usepackage{tikz-cd}
\usepackage{enumitem}
\usepackage{graphicx}
\usepackage{mathdots}
\usepackage{color}
\usepackage{dsfont}
\usepackage{comment}
\usepackage{pdfsync}
\usepackage{epsfig}
\usepackage[all]{xy}
\usepackage{todonotes}
\usepackage{enumitem}

\definecolor{rltblue}{rgb}{0,0,0.4}
\definecolor{drkred}{rgb}{0.6,0,0}
\definecolor{drkgreen}{rgb}{0,0.4,0}

\usepackage[colorlinks=true, urlcolor=rltblue, citecolor=drkgreen, linkcolor=drkred]{hyperref}

\DeclareMathAlphabet{\mymathbb}{U}{BOONDOX-ds}{m}{n}

\newtheorem{thm}{Theorem}[section]
\newtheorem{theorem}[thm]{Theorem}
\newtheorem{lemma}[thm]{Lemma}
\newtheorem{proposition}[thm]{Proposition}
\newtheorem{corollary}[thm]{Corollary}

\theoremstyle{definition}
\newtheorem{definition}[thm]{Definition}

\theoremstyle{remark}

\newtheorem{remark}[thm]{Remark}

\newtheorem{historic}[thm]{Historic Remark}

\theoremstyle{plain}

\newcounter{contenumi}

\def\la{\langle}
\def\ra{\rangle}

\newcommand{\N}{\numberset{N}}

\newcommand{\Sub}{\operatorname{Sub}}
\renewcommand{\S}{\operatorname{S}}

\def\X{\mathcal{X}}

\def\om{\omega}
\def\a{\alpha}
\def\b{\beta}
\def\G{\Gamma}
\def\g{\gamma}
\def\d{\delta}

\def\z{\zeta}
\def\1{\mathsd{1}}

\def\pbar{\bar{p}}

\def\conc{{}^\smallfrown}

\def\ggeq{\mathrel{\geq\!\!\!\geq}}

\makeatletter
\newcommand{\dotminus}{\mathbin{\text{\@dotminus}}}
\newcommand{\@dotminus}{%
	\ooalign{\hidewidth\raise1ex\hbox{.}\hidewidth\cr$\m@th-$\cr}%
}
\makeatother

\newcommand{\RCA}{\mathsf{RCA}_0}
\newcommand{\WKL}{\mathsf{WKL}_0}
\newcommand{\ACA}{\mathsf{ACA}_0}
\newcommand{\ATR}{\mathsf{ATR}_0}
\newcommand{\RT}{\mathsf{RT}}

\newcommand{\TJ}{\mathsf{TJ}}
\newcommand{\lead}{\operatorname{lead}}
\newcommand{\CA}{\mathsf{\text{-}CA}_0}

\def\TJ{{\textrm{TJ}}}
\newcommand{\set}[1]{{\{{#1}\}}}
\def\stops{{\downarrow}}
\def\notstops{{\uparrow}}
\def\WO{{\textrm{WO}}}

\newcommand{\restricted}{{\upharpoonright}}
\newcommand{\leo}{<\sb{o}}
\newcommand{\leqo}{\leq\sb{o}}

\newcommand{\geqo}{\geq\sb{o}}
\newcommand{\pluso}{+\sb{o}}
\newcommand{\gd}[1]{{\ulcorner{#1}\urcorner}}
\def\N{{\mathbb{N}}}
\def\stops{{\downarrow}}
\def\notstops{{\uparrow}}

\newcommand{\homo}{homogeneous}

\def\ifff{\Longleftrightarrow}
\newcommand{\isfunc}[3]{{{#1}\colon{#2}\rightarrow{#3}}}

\usepackage{xcolor}
\newcommand{\lorenzo}[1]{\textcolor{blue}{#1}}
\newcommand{\andrea}[1]{\textcolor{green}{#1}}

\title{The strength of Ramsey's theorem for $\a$-large sets}

\author{Lorenzo Carlucci}
\address{Department of Mathematics, Sapienza University of Rome, Italy}
\email{lorenzo.carlucci@uniroma1.it}
\urladdr{https://sites.google.com/uniroma1.it/lorenzocarlucci-sapienza/homepage}

\author{Andrea Volpi}
\address{Department of Mathematics, University of Udine, Italy}
\email{andrea.volpi@uniud.it}
\urladdr{https://andreasdfghj.github.io/andreavolpi/}

\author{Konrad Zdanowski}
\address{Institute of Computer Science, Cardinal Stefan Wyszynski University, Warsaw, Poland}
\email{k.zdanowski@uksw.edu.pl}
\urladdr{https://www.impan.pl/~kz/}

\subjclass[2020]{03B30 (primary), 03D55, 05D10 (secondary)}
\keywords{Reverse mathematics, Ramsey theory, largeness notions}

\begin{document}

\begin{abstract}
We calibrate the reverse mathematical strength of a family of extensions of Ramsey's theorem to finite colorings of certain subsets of the natural numbers of unbounded finite dimension. 
Specifically, we analyze the principles $\RT^{!\alpha}_k$ asserting that every  $k$-coloring of the exactly $\alpha$-large subsets of an infinite $X \subseteq \mathbb{N}$ admits an infinite homogeneous set, where $\alpha$-largeness is defined via systems of fundamental sequences in the style of Ketonen and Solovay. 
For each countable ordinal $\alpha < \Gamma_0$ and each $k \geq 2$, we prove over $\RCA$ that the hierarchy of theorems $\RT^{!\a}_k$ corresponds exactly to the hierarchy of systems axiomatized by closure under transfinite Turing jumps, yielding a fine-grained classification between $\ACA$ and $\ATR$. 
Our results extend previous work on the case $\alpha=\omega$ and provide a uniform correspondence between countable indecomposable ordinals below $\Gamma_0$ and natural Ramsey-like theorems.
\end{abstract}

\maketitle


%
%
%
%

\section{Introduction}

The study of the logical strength of Ramsey-like theorems has long been central in reverse mathematics. 
In this paper we calibrate the axiomatic strength of a family of far-reaching generalizations of the classical infinite Ramsey's theorem to colorings of sets of unbounded finite dimension. 
In particular we analyze Ramsey-like theorems for colorings of so called exactly $\alpha$-large or $\a$-size subsets of the natural numbers. 

A motivation to generalize largeness notions for finite sets of natural numbers comes from the celebrated result of Paris and Harrington. 
They proved in \cite{parisharrington} that a certain statement in finite Ramsey theory, expressible in Peano arithmetic, is not provable in this system.
This is often regarded as the first \lq\lq natural\rq\rq\ true statement independent from Peano arithmetic.
The concept of a relatively large (or $\omega$-large) finite subset of the natural numbers is the basic ingredient of the Paris-Harrington principle. 
In this context, a finite set $s$ of natural numbers is called relatively large (or $\omega$-large) if $|s| \geq \min s + 1$. 
With this terminology, the Paris-Harrington principle is just the standard finite Ramsey's theorem with the extra condition that the monochromatic set is relatively large. 
Thus, the largeness of the homogeneous set is given in terms of the elements of the set and not as a fixed prescribed cardinality. 
It is possible to generalize this idea to obtain largeness notions associated to ordinals. 
Largeness notions of this kind, based on systems of fundamental sequences,  were introduced in \cite{ketonensolovay}. 
They have since been used by many other authors, especially in the study of  finite Ramsey-like statements \cite{bigorajskaKotlarski1999, bigorajskaKotlarski2002, bigorajska2006, kotlarskiPiekartWeiermann2007, kotlarskiZdanowski2009}. 
Recently, a far-reaching extension of the Paris-Harrington principle has been introduced and analyzed in \cite{marcone2025}. 

Relatively large sets also naturally arise in Ramsey theory for purely combinatorial reasons. 
The natural generalization of Ramsey's theorem to finite colorings of {\em all finite sets} fails, as witnessed by coloring according to the parity of the size of the set.
The following weakening is also false: for every finite coloring $c$ of the finite subsets of the natural numbers there exists an infinite set $H$ of natural numbers such that for infinitely many numbers $n$, $c$ is constant on the $n$-size subsets of $H$. 
Interestingly, a counterexample is given by the coloring that assigns one color to all relatively large sets and the opposite color to all other sets.

While $\omega$-large sets provide a counterexample to the natural extension of Ramsey's theorem to colorings of all finite sets, finite sets $s$ such that $|s| = \min s +1$, called {\em exactly} $\omega$-large sets or {\em $\omega$-size} sets, are the key to define extensions of Ramsey's theorem for colorings of families of finite sets containing elements {\em of unbounded size}.
Weakening the requirement of homogeneity from all finite sets to all exactly $\omega$-large sets results in a true principle, which we denote by $\RT^{!\omega}$ following \cite{carlucci2014}. 
This principle is the base case of a generalization of Ramsey's theorem to colorings of exactly $\a$-large sets for ordinals $\a > \om$, due to Pudl\'ak and R\"{o}dl~\cite{Pud-Rod:82} and by Farmaki and Negrepontis~\cite{Far-Neg:08}, which we denote by $\RT^{!\a}$. 

$\RT^{!\om}$ is known to be computationally and proof-theoretically stronger than the usual Ramsey's theorem for each fixed finite dimension $\RT^n_k$ and even stronger than Ramsey's theorem for all finite dimensions $\forall n \RT^n_k$ (see \cite{clote1984recursion}). 
From the point of view of computability theory, the theorem corresponds the $\omega$-th Turing jump; in reverse mathematics terms, it is equivalent to $\ACA^+$ over $\RCA$, where $\ACA^+$ extends $\RCA$ by the axiom of closure under the $\omega$-th Turing jump (see \cite{carlucci2014}).

Our main goal is to extend \cite{carlucci2014} to all countable ordinals $\alpha < \G_0$ and to prove the following Main Theorem.

\begin{theorem}[Main theorem]\label{thm:main}
For each $\a < \G_0$, for all $k\geq 2$,
$$\RCA \vdash \RT^{! \a}_k \leftrightarrow \Pi^0_{\lead(\alpha)} \CA.$$
\end{theorem}

Here $\lead(\a)$ denotes the first term in the Cantor normal form of $\a$.
The system $\Pi^0_{\b} \CA$ is axiomatized by closure under the $\b$-th Turing jump. 
These systems form a hierarchy intermediate between $\ACA$ and $\ATR$. 
The rest of the notation will be introduced in the next sections.
The Main Theorem shows that for each indecomposable ordinal $\om^\a < \G_0$, the system $\Pi^0_{\om^\a} \CA$ is equivalent to countably many Ramsey-like statements.


Ramsey's theorem for $\a$-size sets $\RT^{!\a}$ can be seen as particular cases of Nash-Williams' generalization of Ramsey's theorem to families of finite subsets of the natural numbers satisfying some specific properties and called {\em barriers} \cite{Todorcevic+2010}. 
Nash-Williams' Ramsey's theorem for barriers is, in turn, a consequence of the clopen Ramsey's theorem (see~\cite{simpson}).
In this context, the family of exactly $\omega$-large sets is known as the Schreier barrier \cite{Todorcevic+2010}. 
More generally, $\a$-size sets are barriers under mild assumptions (see \cite{marcone2025} for details). 

Clote \cite{clote1984recursion,clote1986generalization} gave a fine-grained analysis of the effective content of Nash-Williams' Ramsey's theorem for barriers as a function of the order-type of the barrier under lexicographic ordering. 
From the viewpoint of computability theory, Clote's results are comparable with ours. 
On the other hand, our reverse mathematics results are new and are not trivial consequences of Clote's results. 
The main differences between Clote's work and the present work are the following. 
First, Clote proves the computability-theoretic lower bounds with respect to a notion of \lq\lq canonical barrier\rq\rq\ that is similar to, but does not coincide with, the notion of $\a$-size set. 
More importantly, Clote's lower bound proofs are not formalizable in the setting of subsystems of second-order arithmetic to the extent that they are based on a generalized notion of limit taken along a (canonical) barrier. 
Indeed, the lower bounds need to be proved in a very different way in order to obtain reverse mathematics results. 
Note that, in general, the following is the case. 
Let $P$ be a $\Pi^1_2$-principle such that there exists a computable instance of $P$ such that all of its solutions compute the $\alpha$-th jump $\emptyset^{(\a)}$. 
This does {\em not} imply, in general, that $\RCA + P$ proves $\forall X \exists Y (Y = X^{(\a)})$, even if the proof relativizes. 

We now describe the structure of the paper.
In Section \ref{sec:preliminaries} we introduce the basic notions needed throughout the paper.
In particular we introduce a system of fundamental sequences; we fix a definition of transfinite Turing jumps; we introduce the crucial largeness notions based on ordinals; we define the axiomatic systems $\Pi^0_{\alpha} \CA$; we formulate the Ramsey-like principles of interest.

In Section \ref{sec:lower} we deal with the first half of the Main Theorem \ref{thm:main}, i.e.\ we show that $\RT^{! \om^\a}$ implies $\Pi^0_{\om^\a} \CA$. 
We give two proofs of this result, based on quite different ideas. 
In the first proof we define computable colorings of $\a$-size sets so as to ensure that all of their homogeneous sets code the appropriate transfinite Turing jump.
In the second proof we show that that $\RT^{! \om^\a}_4$ proves a well-ordering preservation principle known to be equivalent to $\Pi^0_{\om^\a} \CA$. 
This proof combines the strategy of \cite{carlucci2025} with recent results of \cite{marcone2025}.

In Section \ref{sec:upper} we conclude the proof of the Main Theorem \ref{thm:main} by showing that for each $\a$, the $(\a+1)$-th Turing jump of an instance of $\RT^{!\a}_k$ computes a solution.

\section{Preliminaries}\label{sec:preliminaries}

We work with the well-known theories of second-order arithmetic $\RCA$, $\ACA$, $\ATR$ and with the less studied intermediate systems $\Pi^0_\a \CA$ for $\a < \G_0$ that we define below.
For background on reverse mathematics we refer to \cite{simpson, dzhafarovmummert}.
For background on fundamental sequences and functions in the Hardy hierarchy we refer to \cite{pohlers, fairtlough}. 
We also use some results from \cite{marconemontalban}, which brings together some aspects of reverse mathematics, notations for well-orderings, and functions in the Hardy hierarchy.

We start by fixing some notation.
As usual, we denote ordinals by lowercase Greek letters.
We use uppercase letters $X, Y, Z$ to denote subsets of $\N$ which may be finite or infinite.
We use lowercase letters $s, t, u$ to denote finite subsets of $\N$. 
We use $\N^+$ to denote the set of positive integers. 

If $X \subseteq \N$ we denote by $[X]^{< \omega}$ the set of finite subsets of $X$ and, for $n \in \N$, by $[X]^n$ the set of subsets of $X$ with exactly $n$ elements.

We identify a subset $X$ of $\N$ with the strictly increasing sequence (finite or infinite) which enumerates it and denote by $X(i)$ or $X_i$ the element in position $i$ in the sequence.
If $s$ is such a finite sequence then $|s|$ denotes its length (coinciding with the cardinality of the set) and we write $s = \langle s_0, \ldots, s_{|s|-1} \rangle$. 
We ambiguously use the same notation to refer to the numerical code of the sequence with respect to a standard pairing function.

Given $s,t \subseteq \N$, by $s \sqsubseteq t$ we mean that $s$, as a sequence, is an initial segment (or prefix) of $t$.
This is stronger than $s \subseteq t$, which denotes set-theoretic inclusion as usual.
The irreflexive versions of the previous relations are denoted by $\sqsubset$ and $\subset$.

We let $s^* = s \setminus \{\max s\}$ and $X^- = X \setminus \{\min X\}$.
By $s < X$ we mean that each element of $s$ is strictly smaller than each element of $X$ or equivalently (when $s$ and $X$ are nonempty) $\max s < \min X$.
We write $s \conc X$ for the concatenation of the sequences $s$ and $X$.
Notice that in our setting $s \conc X$ does make sense only if $s < X$ and, set-theoretically, coincides with $s \cup X$.

\subsection{Ordinals and fundamental sequences}\label{ssec:ordinals}

The \textit{Feferman–Schütte ordinal} $\G_0$ is the first fixed-point of the binary Veblen function, namely the least ordinal $\a$ such that $\a = \varphi_\a (0)$.
It is the proof-theoretic ordinal of the system $\ATR$ (see \cite{simpson} and references therein). 

Ordinals below $\G_0$ can be represented in a standard way using the binary Veblen function. 
We refer to \cite{pohlers} for a detailed treatment of such notations.
This gives a \textit{recursive system of ordinal notation} for ordinals below $\G_0$ and from now on we tacitly assume to work with ordinals in such an ordinal notation system. 
A specific notation system will be defined later.
We use $\le$ to refer to the ordering relation and we write $\alpha + 1$ to denote the successor of $\alpha$ in the ordinal notation system.

We express the fact that $\a$ is a (notation for a) well-ordering as follows. 

\begin{definition}\label{wellorderingdef}
Let $\a$ be a (notation for a) linear ordering $(X, \leq_X)$. 
We say that $\a$ is a \textit{well-ordering} if the following holds
$$\forall f\colon\N\to\a\, (\forall i<j( f(j)\le_X f(i)) \rightarrow \exists n\forall i\geq n (f(i)=f(n))).$$
We denote the above formula by $\WO(\a)$.
\end{definition}

If $\WO(\a)$ is provable in a theory $T$ then $\a$ is a well-ordering in (all models of) $T$ and we say that $\a$ {\em is a well-ordering in} $T$.

Recall that a {\em fundamental sequence} for an ordinal $\a$ is a non-decreasing sequence converging to $\a$.
If we have a fundamental sequence for each ordinal below or equal to $\G_0$, we speak of a {\em system of fundamental sequences on} $\G_0$.
We use the \textit{system of fundamental sequences} on $\G_0$ defined in \cite{marcone2025} and the associated notion of normal form.

For ordinals $\a$ and $\b$ we write $\a \ggeq \b$ to mean that the last term in the normal form of $\a$ is greater than or equal to the first term in the normal form of $\b$.

\begin{definition}\label{system of fund seq}
We stipulate that the fundamental sequence of $0$ is $0[n] = 0$ for all $n$.
For a positive $\a < \G_0$, we distinguish cases.
\begin{enumerate}
	\item If $\a = \a_0 + \om^{\a_1}$ with $\a_0 \ggeq \om^{\a_1}$ then $\a [n] = \a_0 + (\om^{\a_1} [n])$.
	\item If $0 < \a < \om^\a$ then $\om^\a [n] = \om^{\a [n]} \cdot n$.
	\item If $0 < \d < \varphi_\d (0)$ then $\varphi_{\d} (0) [n] = \varphi_{\d [n]}^{n + 1} (0)$.
	\item If $0 < \a < \varphi_\d (\a)$ and $0 < \d$ then $\varphi_{\d} (\a) [n] = \varphi_{\d [n]}^{n+1}(\varphi_{\d}(\a [n]) + 1)$.
    \item $\G_0 [n] = \varphi_{\varphi_{\ddots \varphi_0 (0) \iddots} (0)} (0)$, where $\varphi$ is iterated $n+1$ times.
\end{enumerate}
\end{definition}

A crucial property satisfied by this system of fundamental sequences is {\em nestedness}, which is a weakening of the Bachmann property.
A system of fundamental sequences is {\em nested} if it is never the case for $\g < \b$ and $n > 1$ that $\g > \b[n] > \g[n]$.
 
\begin{theorem}\label{nested}\cite[Theorem 2.16]{marcone2025}
The system of fundamental sequences of Definition \ref{system of fund seq} is nested.
\end{theorem}

For $s\in[\N]^{<\om}$ and $\a < \G_0$ 
we denote by $\a[s]$ 
the ordinal $\a[s_0] \ldots [s_{|s|-1}]$.

\subsection{Coding Turing jumps}\label{subsec: turing jump}

We use standard notation from computability theory. 
For a (finite or infinite) sequence $s$ of natural numbers, $\{e\}^s(n)\downarrow=m$ (or $\{e\}^s(n)=m$) denotes the fact that the computation of the oracle Turing machine with code $e$ on input $n$, run with oracle $s$, terminates and outputs $m$. 
If this computation does not halt we write $\{e\}^s(n)\uparrow$, otherwise we write $\{e\}^s(n)\downarrow$. 
We write $\{e\}^s_t(n)\downarrow$ (or $\{e\}^s(n)\downarrow\leq t$) if the computation halts in less than $t$ steps. 
We write $\{e\}^s(n)\downarrow_m$ to denote that the computation has code below $m$, relative to a fixed coding scheme for computations.  

We code Turing jumps as follows, by transfinite induction.
We tacitly use the fact that the ordinal $\a$ is coded by a natural number.

\begin{definition}\label{dfn:TuringJump}
We denote by $\TJ(X,0)$ the $0$-th jump of $X$, which is the set of pairs $\set{(0,x)\colon x\in X}$.
The $(\alpha+1)$-th jump of $X$, $\TJ(X,\alpha+1)$, is the set 
$$\TJ(X,\alpha)\cup\set{(\alpha+1,e)\colon \set{e}^{\TJ(X,\alpha)}(0)\downarrow }.$$
For a limit $\lambda$, we set $\TJ(X,\lambda)$ as $\bigcup_{\alpha<\lambda}\TJ(X,\alpha)$.
\end{definition} 

Such a coding allows to compute easily $\TJ(X,\beta)$ from $\TJ(X,\alpha)$, for $\beta<\alpha$.
Moreover, it is easy to see that there is a first-order formula that decides whether a given set $Y$ is an $\alpha$-th Turing jump of $X$.
Let $Y\restriction \gamma$ be a formula which restricts $Y$ to elements $x$ of the form $(\delta,z)$, for $\delta\le\gamma$, for example $x=(\delta,z)\land \d \le \gamma \land x \in Y$.
Now, we can write the formula stating that $Y=\TJ(X,\alpha)$ as follows 
$$\forall x\,((0,x)\in Y \leftrightarrow x\in X)\land
\forall \delta\le \a\forall x\,((\delta+1,x)\in Y \leftrightarrow \set{x}\sp{Y\restriction\delta}(0)\stops),$$
where we assume that $\set{x}\sp{Y\restriction\delta}(0)\stops$ is properly expressed.
Given a recursive ordinal notation system, the above formula is $\Pi^0_2$. In what follows, we will use the fact that $Y=\TJ(X,\alpha)$ is definable in first-order arithmetic.

\subsection{Largeness notions based on ordinals}\label{subsec: largeness}

We next define the crucial notion of largeness based on ordinals. 

\begin{definition}\label{defi:largeness}
Let $s \in [\N]^{< \om}$ and $\a < \G_0$.
We say that $s$ is \textit{$\a$-large} if $\a[s] = 0$, \textit{$\a$-small} if $\a[s] > 0$ and \textit{$\a$-size} (or exactly $\a$-large) if $s$ is $\a$-large but $s^*$ is $\a$-small.
\end{definition}

It is clear that for each $\a$ every infinite set has an initial segment which is $\a$-large. For $n \in \mathbb{N}$, a set $s\subseteq\N$ is $n$-size precisely when its cardinality is $n$, so that $[X]^{!n}$ coincides with $[X]^n$.

It was proved in \cite{marcone2025} that for each $s, t \in [\N]^{< \om}$ such that $s \subseteq t$ and each $\a < \G_0$, if $s$ is $\a$-large, then $t$ is $\a$-large, while if $s \subsetneq t$ and $t$ is $\a$-size, then $s$ is $\a$-small. Nestedness is crucial to ensure these properties.

For a set $X\subseteq \N$, we write $[X]^{!\alpha}$ to denote the set of subsets of $X$ which are $\alpha$-size and by $[X]^{! \ge \a}$ the set of subsets of $X$ which are $\a$-large.
As any $\alpha$-size set is finite, $[X]^{!\alpha}$ and $[X]^{! \ge \a}$ are coded as second-order objects, as in our base theory $\RCA$ we have the totality of the exponential function.

We also introduce a sum operation $\uplus$ between largeness notions.

\begin{definition}
We say that $s \in [\N]^{< \om}$ is $(\alpha \uplus \beta)$-large if $s$ can be partitioned in two parts $s_\beta < s_\alpha$ such that $s = s_\beta \conc s_\alpha$, $s_\beta$ is $\beta$-size and $s_\alpha$ is $\alpha$-large.
We may also say that $s$ is $(\alpha \uplus \beta)$-size if in addition $s_\alpha$ is $\alpha$-size.
\end{definition}

Given ordinals $\beta < \alpha$, an $\alpha$-large set need not be $\beta$-large. 
To address this issue, Ketonen and Solovay introduced in \cite{ketonensolovay} the relation $\Rightarrow_n$ and the norm function, and showed that these have well-behaved structural properties for ordinals below $\epsilon_0$. 
As a consequence, any $\alpha$-large set whose minimum exceeds a bound depending only on $\beta$ (called the norm of $\beta$) is guaranteed to be $\beta$-large.
We adapt these ideas to our system of fundamental sequences.
		
\begin{definition}\label{normdefinition}
Given $\a, \b < \G_0$ and $n\in\N$ we write $\a \Rightarrow_n \b$ to mean that $\b = \a[n]\cdots [n]$ for some number (possibly $0$) of $n$'s.

To each ordinal $\a$ we associate the natural number 
$$|\a| = \min \{n > 1 : \G_0 \Rightarrow_n \a \},$$ 
that we call the \textit{norm} of $\a$.
\end{definition}

Notice that we can always assume without loss of generality that $|\a| < |\a + 1|$ for each $\a < \G_0$.

The following is \cite[Proposition 2.12]{marcone2025}.
The proof needs nestedness of the system of fundamental sequences.

\begin{proposition}\label{prop: goodnorm}
For all $\b, \d < \G_0$, if $\b < \d$ then $\b \le \d[|\b|]$.
Moreover, if $s \in [\N]^{<\omega}$, $\d[s] \le \b$ and $\min s \ge |\b|$ then for some $t \sqsubseteq s$ we have $\d[t] = \b$.
Consequently, if $s$ is $\d$-large and $\min s \ge |\b|$ then $s$ is $\b$-large.
\end{proposition}

We comment about the proofs of Theorem \ref{nested} and Proposition \ref{prop: goodnorm} in Remark \ref{technical lemmas} at the beginning of Section \ref{sec:lower}, after introducing all the technical machinery we need.

Note that the minimum in the definition of the norm (Definition \ref{normdefinition}) may not be recursive. However we can get a recursive norm by looking for the first $n$ that we notice gets the job done (which may not be the minimum in the standard ordering). Then the properties of the norm stated in Proposition \ref{prop: goodnorm} are still satisfied by this notion of norm. Thus we make no distinction between the two in what follows.

\subsection{Systems between $\ACA$ and $\ATR$}\label{ssec:systems}

A classification of subsystems of second-order arithmetic that lie strictly between $\ACA$ and $\ATR$ is obtained in \cite{marconemontalban}.
These systems can be characterized by an axiom stating the existence of the $\alpha$-th Turing jump of any given set for $\alpha < \G_0$ and are denoted by $\Pi^0_{\alpha} \CA$.
If $\alpha < \omega$ we simply obtain $\ACA$ while if $\alpha = \omega$ we get $\ACA^+$.
With the formalization of the Turing jump of Subsection \ref{subsec: turing jump} we define $\Pi^0_{\alpha} \CA$ as the theory
$$\RCA + \, \WO(\a) + \, \forall X \exists Y (Y = \TJ(X,\a)).$$

It is easy to see that in the context of reverse mathematics, only jumps corresponding to indecomposable ordinals are meaningful.
Let $\lead(\alpha)$ be the leading term in the Cantor normal form of $\alpha$.

\begin{lemma}\label{equivsystems}
For each $\alpha < \G_0$, $\Pi^0_{\alpha} \CA$ and $\Pi^0_{\lead(\alpha)} \CA$ are equivalent over $\RCA$.
\end{lemma}

\begin{proof}
The left-to-right direction is trivial since $\alpha \ge \lead(\alpha)$ and models of $\RCA$ are closed under Turing reducibility.

For the right-to-left direction let $n$ be a standard natural number such that $\lead(\alpha) \cdot n \ge \alpha$.
Then $\Pi^0_{\lead(\alpha)} \CA$ proves the existence of the $(\lead(\alpha) \cdot n)$-th Turing jump and hence also of the $\alpha$-th Turing jump.
\end{proof}

In view of Lemma \ref{equivsystems} we only study the systems corresponding to $\omega^\alpha$-th Turing jumps for $\alpha < \Gamma_0$.
These subsystems of second-order arithmetic lie strictly between $\ACA$ and $\ATR$ and if $\alpha > \beta$ then $\Pi^0_{\omega^\a} \CA \vdash \Pi^0_{\omega^\b} \CA$.
It is easy to see that the hierarchy is strict.


%
%

\subsection{Ramsey theory}\label{ssec:ramsey theory}

For each $n,k \in \N$ the classical infinite Ramsey's theorem $\RT^n_k$ states that for each coloring $c$ of the $n$-size subsets of an infinite set in $k$ colors, there exists an infinite homogeneous set, namely a set such that all of its $n$-size subsets are mapped to the same color by $c$. 
A purely computability-theoretic analysis of $\RT^n_k$ is given in \cite{jockusch}: for all $n,k\geq 2$, $\RT^n_k$ admits $\Pi^0_n$ solutions and omits $\Sigma^0_n$ solutions; i.e.\ the $(n+1)$th jump of an instance can compute a solution while, for each $n\geq 2$, some computable instance of $\RT^n_2$ admits no $\Sigma^0_n$ solution. 
We refer to an instance of a Ramsey-like theorem to indicate a coloring of the appropriate type and to a solution of a Ramsey-like theorem to indicate a homogeneous set. 
Moreover some computable instance of $\RT^3_2$ {\em codes the jump} in the sense that all of its homogeneous sets compute the first Turing jump.

The above results translate to the framework of reverse mathematics as follows: for each $k\in \N$, $\RT^1_k$ is provable in $\RCA$ while for $n \ge 3$, $\RT^n_k$ is equivalent to $\ACA$. 
Ramsey's theorem for pairs $\RT^2_k$ lies strictly between $\RCA$ \cite{specker1971ramsey} and $\ACA$ \cite{seetapunslaman} and is independent from $\WKL$ \cite{Liu:12}.

We next introduce the generalizations of Ramsey's theorem to colorings of $\a$-size sets that are of central interest for the present paper. They are very close to theorems in \cite{Pud-Rod:82} and \cite{Far-Neg:08}. 

\begin{definition}
Let $\a$ a (notation for a) countable ordinal and let $k\geq 1$. 
$\RT^{!\alpha}_k$ is the following statement: for each infinite $X \subseteq \N$ and each coloring $c \colon [X]^{! \alpha} \to k$ there exists an infinite subset $H\subseteq X$ such that $c$ is constant on $[H]^{!\a}$. 
The set $H$ is called homogeneous or monochromatic for $c$.
\end{definition}

In the previous definition, largeness is understood in the sense of Subsection \ref{ssec:ordinals}, relative to the system of fundamental sequences on $\Gamma_0$ introduced in Definition \ref{system of fund seq}.

As usual in reverse mathematics it is routine to show that for each $\a < \G_0$ and each $j,k \in \N$, the statements $\RT^{! \a}_j$ and $\RT^{!\a}_k$ are equivalent over $\RCA$.

\section{Lower bound}\label{sec:lower}

In this section we establish a lower bound on the logical strength of the principles $\RT^{!\omega^\a}_k$. We give two different proofs. The first one proceeds by designing colorings of the $\alpha$-size subsets of $\N$ with hard homogeneous sets. The second one hinges on the recent \cite{marcone2025} and ideas from \cite{carlucci2025} and proceeds by showing an implication from Ramsey's theorem for $\alpha$-size sets to a well-ordering preservation principle of appropriate logical strength. 


We start with some preliminary lemmas that we need for some technical steps in the rest of the section.

\begin{lemma}\label{wellorder a}
For all $0 < \alpha < \Gamma_0$, for all $k\geq 2$, $\RCA+\RT_k^{! \om^\a} \vdash \WO(\om^\a)$.
\end{lemma}
\begin{proof}
First we claim that for each $\b < \a$, for each $k\geq 2$, $\RCA+\RT_k^{! \om^\a} \vdash \RT_k^{! \om^\b}$.
Let $c \colon [X]^{! \om^\b} \to k$ be a coloring and let $n$ be the norm of $\om^\b$ (recall Definition \ref{normdefinition}).
By Proposition \ref{prop: goodnorm} every $\om^\a$-large set $s$ with $\min s > n$ is also $\om^\b$-large.
Let $d \colon [X \setminus \{0, \ldots, n\}]^{! \om^\a} \to k$ be the coloring defined as $d(s) = c(t)$ where $t$ is the unique $\om^\b$-size initial segment of $s$.
Let $H\subseteq\N$ be an infinite homogeneous set for $d$ of color $j$ for some $j < k$.
We show that $H$ is homogeneous for $c$: if $t \in [H]^{! \om^\b}$ then by definition $c(t) = d(s) = j$ where $s$ is any extension of $t$ to an $\om^\a$-size set in $H$.

We know by inductive hypothesis that for each $\b < \a$, $\RCA+\RT_k^{!\om^\b} \vdash \Pi^0_{\om^\b} \CA$.
By \cite[Theorem 6.16]{marconemontalban}, the system $\Pi^0_{\om^\b} \CA$ is equivalent to a well ordering principle.
Applying \cite[Theorem 1.5]{Arai2020} we get that the proof-theoretic ordinal of this axiom system is exactly $\varphi_{\b+1} (0)$.
We claim that there exists $\b < \a$ such that $\a < \varphi_{\b+1}(0)$.
If not then $\a \geq \varphi_{\b+1}(0)$ for each $\b < \a$ and so $\a \ge \lim_{\b \to \a} \varphi_{\b+1}(0) = \varphi_\a(0)$.
However this contradicts $\a < \G_0$, which is the least $\g$ such that $\g = \varphi_\g(0)$.

Let $\b < \a$ be large enough so that $\om^\a < \varphi_{\b+1}(0)$.
Then we get that $\RT_k^{!\om^\a}$ proves the well-foundedness of every ordinal below $\varphi_{\b+1}(0)$ and in particular it proves that $\om^\a$ is a well-order.
\end{proof}

\begin{corollary}\label{induction}
For all $\a<\G_0$, for all $k\geq 2$, $\RCA+\RT^{! \om^\a}_k$ proves arithmetical transfinite induction over $\om^\a$.
\end{corollary}

\begin{proof}
By \cite[Exercise 5.13.19]{dzhafarovmummert} $\ACA$ proves arithmetical transfinite induction over any well-order.
By Lemma \ref{wellorder a} $\RT^{! \om^\a}_k$ proves that $\om^\a$ is a well-order and by inductive hypothesis it proves $\ACA$ too. 
\end{proof}

\begin{remark}\label{technical lemmas}
As promised at the end of Subsection \ref{ssec:ordinals}, we also comment about the proofs of Theorem \ref{nested} and Proposition \ref{prop: goodnorm}.
The following results can be proved in $\RT^{! \om^\a}_k$ if we restrict the system of fundamental sequences to $\om^\a$, which suffices for us.
\begin{enumerate}
    \item Theorem \ref{nested} states that the system of fundamental sequences of Definition \ref{system of fund seq} is nested.
    Since we only need fundamental sequences up to $\om^\a$, we only need to prove nestedeness on the system induced on $\om^\a$.
    The proof is \cite[Theorem 2.16]{marcone2025}, and uses only arithmetical induction over $\om^\a$, which is available in $\RT^{! \om^\a}_k$ by Corollary \ref{induction}.
    \item Proposition \ref{prop: goodnorm} is proved in \cite[Proposition 2.12]{marcone2025} which only uses nestedness of the system.
    By Theorem \ref{nested} and by what we just noticed, the argument can be carried out in $\RT^{! \om^\a}_k$.
\end{enumerate}
\end{remark}

Before establishing the lower bound, we show that we can restrict ourselves to the study of indecomposable well-orders.

\begin{lemma}[$\RCA$]
$\RT^{! \alpha}_k \vdash \RT^{! \lead(\alpha)}_k$.
\end{lemma}

\begin{proof}
If $\alpha$ is indecomposable then the statement is trivial, so suppose $\alpha > \lead(\alpha)$ and let $\a' < \a$ be such that $\alpha = \lead(\alpha) + \a'$.
Let $c \colon [X]^{! \lead(\alpha)} \to k$.
Notice that by Definition \ref{system of fund seq} each $\alpha$-size set $t$ has a $\lead(\alpha)$-size subset, and so $t$ is $\lead(\alpha)$-large.
Hence $t$ must have a unique $\lead({\alpha})$-size initial segment $s$.  
Consider the coloring $d \colon [X]^{! \alpha} \to k$ defined by $d(t) = c(s)$ where $s$ is the unique $\lead(\alpha)$-size initial segment of $t$.
Then $d$ is a $(c \oplus X)$-computable instance of $\RT^{! \alpha}_k$.
Let $H\subseteq\N$ be an infinite homogeneous set for $d$ and let $h$ be its $\alpha'$-size initial segment.
Then it is straightforward to verify that $H' = H \setminus h$ is an infinite homogeneous set for $c$.
\end{proof}

\subsection{Coding Turing jumps into homogeneous sets}

%
%


In this subsection we work with a generic $\Lambda \le \G_0$ and consider the system of fundamental sequences of Definition \ref{system of fund seq} when restricted to $\Lambda$.
Then we prove the required lower bound in Corollary \ref{lowboundtj}.

We assume we fixed a recursive ordinal notation for $\Lambda$.  
We henceforth refer to ordinals with respect to this notation. 
We use $\leqo,\leo$, etc.\ to refer to the ordering with respect to $\Lambda$.
We write $\alpha\pluso 1$ to denote the successor of $\alpha$ in $\Lambda$, if it exists. 
Similarly, we write $\alpha\pluso n$, for $n\in\N$. 
We write $n\sp{\Lambda}$ to denote the $n$-th element of $\Lambda$, for $n\in \N$.

We want to stress that we see elements of $\Lambda$ in two ways: one related to their position in the ordering $\leqo$ and the other one as natural numbers. 
That is $\a \leo \beta$ means that $\a$ is less than $\beta$ in $\Lambda$ while $\a < \beta$ means that the integer coding $\a$ is less than the integer coding $\beta$. 
Occasionally, when we want to stress this second reading we write $\gd\alpha <\gd\beta$. 
The brackets are to remind that we think about $\alpha$ and $\beta$ as natural numbers. 
However, this is just mnemonic as, in a model of arithmetic, every first order element is a number.
Sometimes we distinguish the ordering we use in a given operation. 
So we write for instance $\min\sb{\leq}$ or $\min\sb{\leqo}$.


We assume that the set of successors in $\Lambda$ is recursive and that a successor of any $\alpha\in\Lambda$ has code greater than $\alpha$, that is $\gd\alpha<\gd{\alpha\pluso1}$.
This ensures that computing a predecessor of an ordinal $\alpha\pluso 1$ is recursive. 
To find a predecessor of $\alpha\pluso 1$ it suffices to find the greatest ordinal below $\alpha\pluso 1$ among numbers $\set{0,\dots,\alpha\pluso 1}$.

From now on we assume that the notion of fundamental sequences is fixed and defined as in Definition \ref{system of fund seq}. 
Let us list all the assumed properties of the notation~$\Lambda$.
\begin{enumerate}
\item $\Lambda$ has a recursive set of successors ordinals,
\item $\gd\alpha<\gd{\alpha\pluso1}$,
\item if $\b \leo \d$, $s \in [\N]^{<\omega}$, $\d[s] \leqo \b$ and $\min s \ge |\b|$ then for some $t \sqsubseteq s$ we have $\d[t] = \b$.
\end{enumerate}

The next step is to define colorings with homogeneous sets of high complexity in the hyperarithmetical hierarchy.
First we define a computable sequence of machines $M_\alpha(y,s)$ for $\a \in \Lambda$, and $s$ a $\alpha$-size sequence. 
We next define a computable sequence of computable colorings $\isfunc{c_\alpha}{[\N]^{!\alpha}}{2}$.
The whole construction is implicitly relativized to an oracle $A$. 
As the oracle $A$ is fixed, we usually omit it in our notation and we write, for instance, $M(x)$ instead of $M\sp A(x)$, for a given machine $M$.

We recall that $\{e\}^X(x)\stops_n$ means that the computation of machine $e$ on input $x$ with oracle $X$ terminates and is coded below $n$. 
This is taken to imply that the use of the computation is also strictly bounded by $n$.

The main property of the construction is that machines $M\sb\alpha(y,s)$ compute in the sense of  Clote's generalized limit \cite{clote1986generalization} the $\alpha$-th Turing jump of $A$. 
However, Clote's generalized limit is not first-order expressible for $\alpha\geqo \omega$. 
The colorings $c\sb\alpha$, for $\alpha\in\Lambda$, are defined in such a way that it is possible to use their \homo\ sets instead of generalized limits. 
Moreover, all this will be provable in a suitable theory of arithmetic.



\begin{remark}\label{pairingfunctionproperty}
We fix a computable pairing function $\langle \cdot, \cdot\rangle : \N\times \N\to \N$. 
Notice that we can assume without loss of generality that $|\gamma| \le \la \gd\g,z \ra$ for all $\gamma \in \Lambda$ and $z\in\N$. 
We write $\langle \g, z\rangle$ for $\langle \gd\g, z\rangle$.
\end{remark}

The first machine $M\sb1(y,x\sb0)$ is defined as follows.
$$
M\sb1(y,x\sb0)=
\begin{cases}
\mbox{Accept} & \mbox{ if } y=\langle 0,z\rangle \mbox{ and } y<x\sb0 \mbox{ and } z\in A,\\
\mbox{Reject} &\mbox{otherwise.}
\end{cases}
$$
The machine $M\sb2(y,\langle x\sb0,x\sb1\rangle)$ is defined as follows, where $x\sb0 < x\sb1$.
$$
M\sb2(y,\langle x\sb0,x\sb1\rangle)=
\begin{cases}
\mbox{Accept} &\mbox{if }y=\langle 0,z \rangle \mbox{ and } 
                                M\sb1(y,x\sb1)\mbox{ accepts,}\\
\mbox{Accept} &\mbox{if } y=\la1,z\ra\mbox{ and }y<x\sb0 \mbox{ and }\set{z}\sp Y(0)\stops\sb{x\sb1} \\
               &\mbox{where } Y=\{y<x\sb 1\colon M\sb0(z,x\sb1)\mbox{ accepts}\},\\
\mbox{Reject} &\mbox{otherwise.}
\end{cases}
$$
The intuition behind the definition of $M\sb{\alpha\pluso1}$ for the successor case is that $M\sb{\alpha\pluso1} (\langle\alpha\pluso1, z\rangle,x\sb0, s)$ accepts, where $\langle x_0\rangle \conc s$ is $(\alpha\pluso1)$-size, if machine $z$ stops below $\min s$ with input $0$ and an oracle $Y_\alpha^s$ computed by $M\sb\alpha$. 
This oracle is intended to approximate the $\alpha$-th jump of $A$. 
Thus, $M\sb{\alpha\pluso1}$ is defined for an $(\alpha\pluso1)$-size sequence $\langle x\sb0\rangle \conc s$  as follows.
$$
M\sb{\alpha\pluso1}(y,\langle x\sb0\rangle \conc s)=
\begin{cases}
\mbox{Accept} &\mbox{if }y=\langle\gamma,z\rangle\mbox{ and } \gamma\leqo \alpha \mbox{ and }
              M\sb\alpha(y,s)\mbox{ accepts,}\\
\mbox{Accept} &\mbox{if }y=\langle\alpha\pluso1,z\rangle\mbox{ and } y<x\sb0 \mbox{ and }
                                    \set{z}\sp {Y^s_\alpha} (0)\stops\sb{\min s},\\
               &\mbox{where } Y^s_\alpha =\{ y<\max s\colon M\sb\alpha(y,s)\mbox{ accepts}\},\\
\mbox{Reject} &\mbox{otherwise.}
\end{cases}
$$
For limit ordinals $\lambda \in \Lambda$ we define
$$
M_\lambda(y,\langle x\sb0\rangle \conc s)=M_{\lambda[x\sb0]}(y,s).
$$

Note that for all $\alpha\in\Lambda$, machine $M_\alpha$ receives an input of the form $\langle y, s\rangle$ where  $s$ is an $\alpha$-size sequence and $y< \min s$.
We implicitly stipulate that machine $M\sb\alpha$ rejects on inputs of different type. 

We next state some properties of machines $M\sb\alpha$. We assume that $\Lambda$ is a recursive notation for a linear ordering and that $\alpha\in\Lambda$. 
If we require $\Lambda$ to be a well ordering we state it explicitly.

First, let us remark that all machines $M\sb\alpha$ are total by construction.

\begin{lemma}[$\RCA+\WO(\Lambda)$]
For each $\alpha\in\Lambda$, $M\sb\alpha$ is total.
\end{lemma}

\begin{lemma}[$\RCA+\WO(\Lambda)$]\label{lem:from:alpha:to:beta}
Let $\beta, \alpha\in\Lambda$ with $\beta\leqo \alpha$.
Let $t$ be a $\beta$-size sequence and let $s$ be such that $s < t$ and $\alpha[s]=\beta$.
Then 
\begin{multline*}
\set{\langle\gamma,z\rangle\colon M\sb\beta(\langle\gamma,z\rangle,t)\mbox{ accepts and }\gamma\leqo \beta } =\\
\set{\langle\gamma,z\rangle\colon M\sb\alpha(\langle\gamma,z\rangle, s \conc t)\mbox{ accepts and }\gamma\leqo \beta }.
\end{multline*}
\end{lemma}

\begin{proof}
Let $\beta,\alpha\in\Lambda$, $s$ and $t=\langle t\sb0,\dots,t\sb k\rangle$ as in the lemma and let $\gamma\leqo\beta$.
If $\langle\gamma,z\rangle$ is accepted by $M\sb\beta(\langle\gamma,z\rangle,t)$ or by$M\sb\alpha(\langle\gamma,z\rangle,s \conc t)$ then there exists $i\leq k$ such that $\gamma=\beta[t\sb 0]\cdots [t\sb {i-1}]= \alpha[s][t\sb 0]\cdots [t\sb{i-1}]$.
So, we have the following equivalences, 
\begin{align*}
M\sb\beta(\langle\gamma,z\rangle,t)\mbox{ accepts} & \ifff M\sb\gamma(\langle\gamma,z\rangle,\langle t\sb i,\dots, t\sb k\rangle)\mbox{ accepts}\\
&\ifff M\sb\alpha(\langle\gamma,z\rangle,s \conc t)\mbox{ accepts}.
\end{align*}
This proves the lemma.
\end{proof}

We now define the colorings $c_\alpha$, for $\alpha \geq \omega$. For $\alpha = \omega$ the construction is close to the one in~\cite{carlucci2014}.

\begin{definition}\label{def:colorings}
For $\alpha \pluso 3\in\Lambda$, we define $\isfunc{c\sb{\alpha \pluso 3}}{[\N]\sp{!(\alpha\pluso3)}}2$ as follows.
$$
c\sb{\alpha \pluso 3}(\langle a_0,a_1,a\sb2 \rangle \conc s)=
\begin{cases}
1 & \mbox{ if } \forall e,x<a_0\ \  \{e\}^{ Y_\alpha^s}(x)\stops_{a_1} \Leftrightarrow \{e\}^{ Y_\alpha^s}(x)\stops_{a\sb 2},\\ 
& \mbox{ where } Y_\alpha^s=\{y\leq \max s\colon M_{\alpha}(y,s)\mbox{ accepts}\},\\
0 &  \mbox{otherwise.}
\end{cases}
$$
For limit ordinals and other ordinals $\gamma$ not covered by the above definition we define $c\sb\gamma$ as the constant function equal to $1$.
\end{definition}

The following counting argument is crucial for the rest of the proof.

\begin{lemma}[$\RCA+\WO(\Lambda)$]\label{lem:colored:one}
Let $\beta \pluso 3\in\Lambda$ and let $H\subseteq \N$ be an infinite homogeneous set for $c\sb{\beta \pluso  3}$. 
Then, $H$ has color $1$.
\end{lemma}

\begin{proof}
Let $\beta$ be given and let $H\subseteq\N$ be an infinite homogeneous set for $c\sb{\beta\pluso3}$.
Let us choose $a\sb0 < a\sb1 < \dots < a\sb{(a\sb0)\sp2+2}\in H$ and let $s$ be a $\beta$-size subset of $H$ such that $a\sb{(a\sb0)^2+2}<\min s$. 

Now, let us observe that the only way for $c\sb\beta(a\sb0, a\sb i, a\sb {i+1}, s)$ to be equal to $0$, for $1\leq i\leq (a\sb0)\sp2+1$, is the existence of $e\sb i,x\sb i<a\sb0$ such that $\set{e\sb i}\sp {Y_\beta^s}(0)\stops \sb{a\sb {i+1}}$ and not $\set{e\sb i}\sp {Y_\beta^s}(x\sb i)\stops\sb{a\sb i}$, where $Y_\beta^s=\set{y<\max s \colon M\sb{\beta}(y,s) \mbox{ accepts}}$.

Let us fix such $(e_i,x_i)$ for each $a_i$, where $1\leq i\leq (a_0)^2+1$.
Since $Y_\beta^s$ is fixed, $(e\sb i,x\sb i)$ is different from $(e\sb j,x\sb j)$, for each $1\leq i<j\leq a\sb0+1$. 
It follows that $c\sb\beta(a\sb0,a\sb1,a\sb 2,s)=1$ since we do not have enough elements below $a\sb 0$ to differentiate all pairs $(a\sb i, a\sb{i+1})$. 
This implies that the color of $H$ is $1$.
\end{proof}

Let us recall that all our constructions are relativized to a fixed oracle $A$.

\begin{lemma}[$\RCA+\WO(\Lambda)$]\label{lem:recursive:stop}
Let $\beta \pluso 3\in \Lambda$ and let $H=\set{h\sb0 < h\sb1 < \dots}$ be an infinite homogeneous set for $c\sb{\beta \pluso 3}$. 
Then, for any $i\in\N$  and for any $e,x<h\sb i$ we have that if $\{e\}^A(x)\stops$ then $\{e\}^A(x) \stops \sb{h\sb{i+1}}$.
\end{lemma}

\begin{proof}
Let $\beta \pluso 3\in \Lambda$ and $e,x<h\sb i$ be as in the statement of the lemma and such that $\set{e}^A(x)\stops$.

Then there exists a $(\beta \pluso 1)$-size sequence $\la a\sb 2, a\sb 3, \dots, a\sb n \ra$ in $H$ such that $a\sb 2 < a\sb 3 < \dots < a\sb n$ such that $h\sb {i+1} < a\sb 2$ and $\set{e}^A(x)\stops\sb{a\sb 2}$. 
In fact, if $\set{e}^A(x)\stops$ then $\set{e}^{A\restricted k}(x)\stops$ for some $k\in\N$. 
Then note that for sufficiently large $s\in [H]^{!\beta}$, $A\restricted k \subseteq Y_\alpha^s$. 

But the sequence $\la h\sb i,h\sb {i+1},a\sb 2,\dots,a\sb n \ra$ in $H$ is $(\beta \pluso 3)$-size and its color is $1$ by Lemma \ref{lem:colored:one}. 
It follows that $\set{e}^A(x)\stops\sb{h\sb {i+1}}$.
\end{proof}

We introduce a computable operation $S(\alpha, X)$ where $X\subseteq\N$ is infinite and $\alpha\in\Lambda$.
The operation returns a new set which is more sparse than $X$. 

\begin{definition}[$\alpha$-scattering of $X$]\label{dfn:set:scattering}
Let $X=\set{x\sb0 < x\sb1 < x\sb2 < \cdots}\subseteq \N$ be an infinite set. 
We define the $\alpha$-scattering of $X$, $$S(\alpha,X)=\{y_0 < y_1 < y_2 < \cdots\},$$ 
as the result of the following procedure.

Take the first element of $X$ and output it as $y\sb0$. 
Assume that we have defined up to $y\sb i$.
Take the first three elements $z\sb 0, z\sb 1, z\sb 2$ above $y\sb i$ and then take the $\a$-size sequence of consecutive elements of $X$ above $z\sb 2$. 
Output the next element of $X$ as  $y\sb{i+1}$.
We also define $S \sp 1 (\alpha,X)= S(\alpha, X)$ and $S \sp {n+1} (\alpha,X)$ as $S(\alpha, S \sp n(\alpha, X))$ for $n>0$.
\end{definition}

It is obvious that if $\Lambda$ is a computable linear ordering such that $\WO(\Lambda)$ then the procedure used to define $S(\alpha, X)$ is computable and we may decide whether $y\in S(\alpha,X)$ given $X$ as an oracle.
It should also be clear that if $\WO(\Lambda)$ and $X\subseteq\N$ is infinite then for all $\alpha\in\Lambda$, $S(\alpha,X)$ is infinite.

The operation of taking $S(\alpha,X)$ is intended to increase the amount of homogeneity of $X$. 
Namely, Lemma \ref{lem:coloring:monotonicity} below shows that if we start with an infinite $c\sb{\alpha \pluso 3}$-homogeneous set $X$ then $S(\alpha,X)$ is almost $c\sb{\beta \pluso 3}$-homogeneous for all $\beta\leqo \alpha$ in the following sense: for each $\beta \leqo \alpha$ there exists an $n=n(\beta)$ such that $S(\alpha, X)\cap [n, \infty)$ is $c_{\beta\pluso3}$-homogeneous. 
Furthermore $n(\beta)$ is uniformly computable from $\beta$. 

Recall that for $X\subseteq\N$ and $n\in\N$ we denote by $X^{>n}$ the set $X \cap [n+1, \infty)$.
The main property of the set $S(\a, X)$ from Definition \ref{dfn:set:scattering} is the following. 
If $x, y \in X$ and $s \in [X]^{<\omega}$ we say that $s$ is in the interval $(x,y)$ if $x < \min s$ and $\max s < y$. 

\begin{lemma}[$\RCA+\WO(\Lambda)$]\label{lem:scattering:properties}
Let $X\subseteq\N$ be infinite.
Let $\alpha\in\Lambda$ and let $S(\alpha,X) = \set{a\sb0 < a\sb1 < a\sb 2 < \cdots}$.
\begin{enumerate}
\item For each $\beta\leqo \alpha$, for each $i\in\N$ the following holds: if $|\beta| \leq a\sb i$, then there exists a set $s$ in $(a\sb i,a\sb{i+1})$ of consecutive elements of $X$ such that $\alpha[s] =\beta$.
\item For each $\gamma\leqo\beta\leqo \alpha$, for each $i\in\N$ the following holds: if $|\beta|,|\gamma|\leq a\sb i$, then there exists a set $s$ in $(a\sb i, a\sb{i+1})$ of consecutive elements of $X$ such that $\beta[s]=\gamma$.
\end{enumerate}
\end{lemma}

\begin{proof}
We prove both points using the properties of fundamental sequences. 
The first point follows easily from Proposition \ref{prop: goodnorm} and the fact that $X$ contains an $\alpha$-size set in $(a\sb i, a\sb{i+1})$ by definition of $S(\alpha, X)$.
As for the second point we take the sequence $t=\langle t\sb0,\dots,t\sb {k-1}\rangle$ witnessing the first point for $\gamma$ and $\alpha$ and, by Proposition \ref{prop: goodnorm} we find its subsequence $\langle t\sb 0,\dots, t\sb {j-1}\rangle$ such that $\alpha[t\sb 0]\cdots [t\sb {j-1}]=\beta$. 
Then, the sequence $\langle t\sb j,\dots, t\sb {k-1}\rangle$ witnesses the second point.
\end{proof}

We introduce some auxiliary functions that are used in the next proof. 
Let us fix a total computable function $f \colon \N \times \N \to \N$ such that for all $e\in\N$, $\beta \in \Lambda$, $f(e,\beta)$ is the code of a Turing machine constructed from $e$ and $\beta$ that acts as $e$ except treating oracle calls as follows: for any oracle call of the form $\la \gamma,z \ra$ the machine $f(e,\beta)$ checks whether $\gamma\leqo \beta$. 
If $\gamma\leqo\beta$, then it simply passes the question to the oracle, if not it assumes that the oracle answer is \lq\lq not\rq\rq. 
Thus, $f(e,\beta)$ with oracle $X$ acts like $e$ with the oracle $\set{\langle \gamma,z\rangle\in X\colon \gamma\leqo \beta}$.

The correspondence between the two machines $e$ and $f_\beta(e)$ can be given by computable translations {\em of their computations}, which we now describe. 

The function $f\sb{\beta\rightarrow \alpha}$ takes a (code of a) computation $C$ of a machine $e$ with oracle $X$ and computes $C'$, a (code of a) computation of $f_\beta(e)$ with oracle $X$.
Concretely, it replaces each oracle query $\langle \gamma,z\rangle$ with a subroutine which checks whether $\gamma\leqo\beta$. 
If $\gamma\leqo\beta$, then $f_\beta(e)$ passes the question to $X$, otherwise, $f\sb \beta (e)$ assumes the answer is \lq\lq not\rq\rq.
It may happen that $\langle \gamma, z \rangle \in X$ and it is not the case that $\gamma\leqo\beta$. 
In this case the computation coded by $f\sb{\beta\rightarrow \alpha}(C)$ stops with no output as there would be inconsistency in behaviors of $e$ and $f\sb \beta(e)$.  

The function $f\sb{\alpha\rightarrow \beta}$ computes from a (code of a) computation $C'$ of $f\sb \beta(e)$ with oracle $X$ a corresponding a (code of a) computation of $e$ with oracle $\{ \langle \gamma,z\rangle \in X \colon \gamma\leqo\beta\}$. 
It replaces the subroutine checking whether $\gamma\leqo\beta$ by a direct oracle call. 

We need the following properties of $S(\alpha,X)$.
Note that we do not quantify over $X$ in the inductive thesis. 
Thus, we treat $X$ as a second order parameter. 
To simplify the notation in the rest of the section, let $n_{\a,\b}$ be the maximum between $|\beta|$, $|\a|$, the code of $f_\b$, the code of $f\sb{\alpha\rightarrow \beta}$ and the code of $f\sb{\b\rightarrow \a}$.


\begin{lemma}[$\RCA+\WO(\Lambda)$]\label{lem:coloring:monotonicity} 
Let $\beta\leqo\alpha\in\Lambda$, let $H\subseteq\N$ be an infinite, $c\sb{\alpha \pluso3}$-homogeneous set. 
Then, $S(\alpha,H)\sp{>n_{\a,\b}}$ is $c\sb{\beta\pluso3}$-\homo.
\end{lemma}

\begin{proof}

Let $\beta\leqo\alpha\in\Lambda$ and $H$ satisfy the assumptions of the lemma.
Let $\la b\sb0,b\sb1,b\sb 2\ra \conc t$ be a $(\beta \pluso 3)$-size sequence in $S(\alpha,H)\sp{>n_{\a,\b}}$. 
Let $e,x<b\sb 0$ be such that 
$$
\set{e}\sp{Y\sb \beta \sp t}(x)\stops\sb{b\sb 2},
\mbox{ where } 
Y\sb \beta \sp t=\set{ y\leq \max(t) \colon M\sb \beta(y,t)\mbox{ accepts}}.
$$
We need to show that $\set{e}\sp{Y\sb \beta \sp t}(x)\stops\sb{b\sb 1}$.

Let $a\sp{1}\sb0,a\sp{2}\sb0, a\sp{3}\sb0\in X$ be such that 
$$
b\sb0<a\sp{1}\sb0<a\sp{2}\sb0 < b\sb 1.
$$
The code of $f\sb{\beta}$ is below $a\sp{2}\sb0$.
Now, let $a\sb 2\in H$ and let $s$ be a sequence in $H$ such that 
$$
b\sb 2<a\sb 2<\min s <\max s <\min t 
$$ and 
$$
\beta=\alpha[s].
$$
The existence of $a_2$ and $s$ is guaranteed by Lemma \ref{lem:scattering:properties}.
Finally, let
$$
Y\sb \alpha \sp {s\conc t} =\set{ y \leq \max t \colon M\sb \alpha(y,s \conc t)\mbox{ accepts}}.
$$
By Lemma~\ref{lem:from:alpha:to:beta} we have the following equality
\begin{equation}\label{eq:Yb:Ya}
Y\sb \beta \sp t = \set{\langle \gamma,z\rangle\in  Y\sb \alpha \sp {s \conc t}\colon \gamma\leqo \beta}.
\end{equation}
Thus, $e$ with oracle $Y\sb\beta \sp t$ acts as $f\sb\beta(e)$ with oracle $Y\sb\alpha \sp {s \conc t}$, i.e., 
$$\forall x \in\N \, (\{e\}^{Y\sb\beta\sp t}(x) \simeq \{f_\beta(e)\}^{Y\sb\alpha\sp {s \conc t}}(x)).$$
The correspondence between the two machines coded by $e$ and $f_\beta(e)$ can be given by recursive translations, $f\sb{\beta \rightarrow \alpha}$ and $f\sb{\alpha \rightarrow \beta}$, of their computations.

Recall that the computable function $f\sb{\beta\rightarrow \alpha}$ takes a (code of a) computation of a machine $e$ with an oracle $X$ and outputs a (code of a) computation of $f\sb{\beta}(e)$ with oracle $X$. 
The main property of $f\sb{\beta\rightarrow \alpha}$ is the following.
\begin{equation}\label{eq:f:beta:alpha}
\begin{split}
\lefteqn{\mbox{If $C$ is a computation of $e$ with oracle $Y\sb \beta \sp t$}}&\\
&\mbox{\ \ \ \ \ \ then $f\sb{\beta\rightarrow \alpha}(C)$ is a computation of $f\sb \beta(e)$ with oracle $Y\sb \alpha \sp {s \conc t}$}.
\end{split}
\end{equation}
Property~\eqref{eq:f:beta:alpha} follows easily from equality~\eqref{eq:Yb:Ya} which ensures that there is no inconsistency between  these two computations.  

It is crucial that the transformation $f\sb{\beta\rightarrow \alpha}$ does not use the oracle since this allows us to apply Lemma~\ref{lem:recursive:stop}. 
More precisely, by our choice of $n_{\a,\b}$ we know that the code of $f\sb{\beta \rightarrow \alpha}$ is $<a\sp{1}\sb 0$.

Now, let $C$ be the computation of $e$ on $x$ with oracle $Y\sb \beta \sp t$. 
Recall that we are reasoning under the assumption $\{e\}^{Y_\beta^t}(x) \stops_{b_2}$, so that $C<b\sb 2$.
Then, $C'=f\sb{\beta\rightarrow \alpha}(C)$ is the computation of $f\sb \beta(e)$ on $x$ with oracle $Y\sb \alpha \sp {s \conc t}$. 
By Lemma~\ref{lem:recursive:stop} and the above observations on $f\sb{\beta \rightarrow \alpha}$, we have $C'<a\sb 2$. 
Since the sequence $\la a^2_0, a^3_0, a_2\ra \conc s \conc t$ is $(\alpha \pluso 3)$-size and in $H$, it has to be the case that 
$$ c_{\alpha \pluso 3}(\la a^2_0, a^3_0, a_2\ra \conc s \conc t) = 1,$$
by Lemma \ref{lem:colored:one}. 
Since $f_\beta(e), x < a^2_0$ this implies, in particular, that
$$\{f_\beta(e)\}^{Y_\a^{s\conc t}}(x)\stops_{a_2} \Rightarrow \{f_\beta(e)\}^{Y_\a^{s\conc t}}(x)\stops_{a^3_0}.$$
Therefore $C'<a\sp{3}\sb 0$. 

It should be obvious that the other operation, $f\sb{\alpha \rightarrow \beta}$, which computes from a computation $C'$ of $f\sb \beta(e)$ with oracle $X$ a corresponding computation of $e$ with oracle $\set{(\gamma,z)\in X\colon \gamma\leqo \beta}$ satisfies the following property:
\begin{equation}\label{eq:f:alpha:beta}
\begin{split}
\lefteqn{\mbox{If $C'$ is a computation of $f\sb \beta(e)$ with
oracle $Y\sb \alpha \sp {s \conc t}$}}&\\ 
&\mbox{\ \ \ \ \ \ then $f\sb{\alpha\rightarrow \beta}(C')$ is a computation of $e$ with oracle $Y\sb \beta \sp t$.}
\end{split}
\end{equation}
Again by assumption we have that a code of $f_{\beta \to \alpha}$ is below $n_{\a,\b}$. 
Then, since $C'<a\sp{3}\sb0 < b_1$, Lemma \ref{lem:recursive:stop} ensures that 
$C=f\sb{\alpha\rightarrow\beta}(C')<b\sb 1$ which shows that $\set{e}\sp{Y\sb \beta \sp t}(x)\stops\sb{b\sb 1}$ and ends the proof.
\end{proof}

Now we want to prove that given an infinite $c\sb{\alpha\pluso3}$-\homo\ set $H\subseteq\N$, an ordinal $\beta\leqo \alpha$ and two $\beta$-size sequences $s,t$ from $S(\alpha,H)\sp{>n_{\a,\b}}$ it does not matter which sequence we input to the machine $M\sb\beta$ (Lemma \ref{lem:eq:homo:wit}).
We need some preliminary lemmas.

\begin{lemma}[$\RCA+\WO(\Lambda)$]\label{lem:bounding:computations}
Let $\alpha\in\Lambda$ and let $H\subseteq\N$ be an infinite $c\sb{\alpha \pluso3}$-\homo\ set, let $\beta\leqo \alpha$, let $s$ be a $\beta$-size sequence in $S\sp2(\alpha,H)\sp{>n_{\a,\b}}$ and let $s\sb0, s\sb1$ be in $S\sp2(\alpha,H)\sp{>n_{\a,\b}}$ such that $s\sb0 < s\sb1 < \min s$. 
Then, 
$$\forall e,x<s\sb0 \left(\set{e}\sp {Y \sb \beta \sp s}(x)\stops\sb {\min s}\Longrightarrow \set{e}\sp {Y \sb \beta \sp s}(x)\stops\sb {s\sb1}\right),$$
where $Y_\beta^s =\set{y<\max s\colon M\sb\beta(y,s)\mbox{ accepts}}$. 
\end{lemma}

\begin{proof}
Let $s = \langle s_2,\dots, s_k\rangle$ for some 
$s\sb2 < s\sb3 < \dots < s\sb k$ in $S\sp2(\alpha,H)\sp{>n_{\a,\b}}$ such that $s\sb1 < s\sb2$.
Let $e,x<s\sb0$ be such that 
$$
\set{e}\sp {Y\sb\beta \sp s}(x)\stops\sb {s\sb 2}.
$$
We claim that $\set{e}\sp {Y\sb\beta \sp s}(x)\stops\sb {s_1}$.
We consider two cases according to whether $\beta$ is a successor or a limit ordinal. 

\smallskip

\textbf{Case 1:} $\beta=\delta+1$. 
Choose $s\sp 1\sb 2\in S(\alpha,H)$ such that $s\sb 2< s\sp 1\sb 2<s\sb 3$. 
This is possible by the properties of $S(\alpha, H)$. 
The sequence $\langle s\sp 1 \sb 2\rangle \conc s^- = \langle s\sp 1 \sb 2, s\sb 3,\dots, s\sb k\rangle$ is $(\delta+1)$-size since $\beta=\delta+1$ and $s$ is $\beta$-size by hypothesis. By definition we have
$$ \langle\gamma, z\rangle \in Y^s_{\delta+1} \Longleftrightarrow \langle\gamma,z\rangle < \max s \mbox{ and } M_{\delta+1}(\langle \gamma, z\rangle, s) \mbox{ accepts}.$$

Moreover, for each $\gamma$ and $z$ such that $\gamma\leqo\beta$ and $\la\gamma,z\ra<s\sb 2 = \min s$, by the definitions of $Y^s_{\delta+1}$ and of $M\sb{\delta+1}$, we have that
\begin{align*}
\langle \gamma,z\rangle\in Y\sb{\delta+1} \sp s &\ifff M\sb{\delta+1}(\langle \gamma,z\rangle, s)\mbox{ accepts}\\
&\ifff  M\sb{\delta+1}(\langle \gamma,z\rangle, \langle s\sp 1\sb 2 \rangle \conc s^-)\mbox{ accepts.}
\end{align*}
To show this we observe that for ordinals $\gamma\leo \beta$ and for $\langle \gamma,z \rangle<s\sb 2$ we have 
$$
\langle \gamma,z\rangle\in Y\sb{\delta+1} \sp s \ifff M\sb{\delta}(\langle \gamma,z\rangle, s^-)\mbox{ accepts}
$$
by definition of $M_{\delta+1}$. Moreover, $\langle \delta+1,z\rangle\in Y\sb{\delta+1} \sp s$ if and only if $\set{z}\sp{Y\sb{\delta} \sp{s^-} }(0)\stops\sb{s\sb3}$, where, as usual, we set
$$
{Y\sb{\delta}\sp{s^-}}=\set{\langle \gamma,e\rangle < \max(s^-)\colon M\sb{\delta}(\langle \gamma,e\rangle,s^-)\mbox{ accepts.}}
$$ 
Then note that the right hand side of the above equality does not depend on $s\sb 2$ or $s\sp 1\sb2$.
It follows that if the pair $\langle e,x\rangle<s\sb0$ is such that 
$$
\set{e}\sp {Y\sb{\delta}\sp{s^-}}(x)\stops\sb {s\sb 2}\mbox{ and not }
\set{e}\sp {Y\sb{\delta}\sp{s^-}}(x)\stops\sb {s\sb 1}.
$$
then $\langle e,x\rangle$ witnesses that the $(\beta \pluso 3)$-size sequence $\langle s\sb0,s\sb 1,s\sb 2, s\sp 1\sb 2,s\sb 3,\dots,s\sb k\rangle$ in $S(\alpha,H)$ is colored $0$ by $c\sb{\beta \pluso 3}$. 
This is a contradiction since $n_{\a,\b}<s\sb0$, $S(\alpha,H)^{>n_{\a,\b}}$ is an infinite homogeneous set for $c\sb{\beta\pluso3}$ by Lemma \ref{lem:coloring:monotonicity} and any such set above $n_{\a,\b}$ has to be colored $1$ by Lemma \ref{lem:colored:one}.

\smallskip

\textbf{Case 2:} $\beta$ is a limit ordinal. 
Recall that the set $Y^s_\beta$ involved in the statement in this case is 
$$Y^s_\beta  = \{ y < \max s \colon M \sb{\beta[s_2]}(y, s^-) \mbox { accepts}\}$$
by definition of $M_\beta$. 

We choose $s\sp 1\sb 2\in S(\alpha,H)$ and a (possibly empty) sequence $\tilde{s}$ in $S(\alpha,H)$ such that $s\sb2<s\sp 1\sb 2 <\min \tilde{s}$ and $\max \tilde{s} <s\sb 3 =\min(s^-)$ and 
$$
\beta[s\sp 1\sb 2][\tilde{s}]=\beta[s\sb 2].
$$
The last property ensures that the sequence $\langle s^1_2\rangle\conc\tilde{s}\conc s^-$ is $\beta$-size. 
Now, let 
$$
\widetilde{Y} = {Y\sb\beta \sp {\langle s^1_2\rangle\conc\tilde{s}\conc s^-}}=\set{y<\max s\colon M\sb\beta(y,\langle s^1_2\rangle\conc\tilde{s}\conc s^-)\mbox{ accepts}}.
$$
We need to show that below $s\sb 2$ the sets $\widetilde{Y}$ and $Y\sb\beta\sp s$ coincide.
If we succeed then we get that $\set{e}\sp {\widetilde{Y}}(x)\stops\sb{s\sb 2}$ from the hypothesis that $\{e\}^{Y^s_\beta}(x) \stops_{\min s}$. 
Recall that by Lemma \ref{lem:coloring:monotonicity} $S(\alpha, H)\sp{>n_{\a, \b}}$ is $c_{\beta\pluso3}$-homogeneous and by Lemma \ref{lem:colored:one} it is colored $1$. 
It follows then that $\set{e}\sp {\widetilde{Y}}(x)\stops\sb{s\sb 1}$ since otherwise the color of the $(\beta\pluso3)$-size sequence $\langle s\sb0,s\sb1, s\sb 2,s\sp 1\sb 2\rangle\conc \tilde{s}\conc s\sp -$ would be $0$ under $c\sb{\beta \pluso 3}$. 

To complete the proof let us take a pair $\langle \gamma,z\rangle <s\sb 2$ with $\gamma\leqo\beta$.
Since $|\gamma|< s\sb 2$ by Remark \ref{pairingfunctionproperty} about the choice of the pairing function, it follows that if $\g \leo \b$ then $\gamma\leqo \beta[s\sb 2]$.
Since $\beta[s\sp 1\sb2][\tilde{s}]= \beta[s \sb 2]$, and $\langle s^1_2\rangle \conc \tilde{s}\conc s^{-}$ is $\beta$-size, then there exists $3\leq j\leq k$ such that 
$$
\gamma=\beta[s\sp 1\sb 2][\tilde{s}][s\sb 3]\cdots [s\sb j] = \beta[s_2][s_3] \cdots[s_j].
$$
Then
\begin{equation*}
\begin{split}
\langle \gamma,z\rangle\in \widetilde{Y} 
& \ifff M\sb{\beta}(\langle \gamma,z\rangle, \langle s\sp 1\sb 2\rangle\conc\tilde{s}\conc s\sp{-})\mbox{ accepts}\\
& \ifff M\sb{\gamma}(\langle \gamma,z\rangle, \langle s\sb {j+1},\dots,s\sb k\rangle)\mbox{ accepts}\\
& \ifff  M\sb{\beta}(\langle \gamma,z\rangle,\langle s\sb 2\dots,s\sb k\rangle)\mbox{ accepts}\\
&\ifff \langle \gamma,z\rangle\in {Y\sp s\sb\beta}.
\end{split}
\end{equation*}
The first and last of the above equivalences are by definition, the second and the third are by Lemma \ref{lem:from:alpha:to:beta}.
This ends the proof.
\end{proof}

The next lemma guarantees a crucial indiscernibility property of $c_{\alpha\pluso3}$-homogeneous sets.

\begin{lemma}[$\RCA+\WO(\Lambda)$]\label{lem:lem:eq:homo:wit}
Let $H\subseteq\N$ be an infinite $c\sb{\alpha \pluso3}$-\homo\ set. 
Then, for each $\beta\leqo \alpha$ and for each choice of $\beta$-size sequences $s,t$ in $S\sp2(\alpha,H)\sp{>n_{\a,\b}}$ such that $s < t$, the following equivalence holds
$$
M\sb\beta(\langle \gamma,e\rangle,s)\mbox{ accepts} \ifff
M\sb\beta(\langle \gamma,e\rangle,t)\mbox{ accepts},
$$
for all $\langle \gamma,e\rangle<\min s$.
\end{lemma}

\begin{proof}
Let us assume that we have $\beta\leqo\alpha$
and $k, n \in \N^+$ such that $s=\langle a\sb0,\dots,a\sb {k-1}\rangle$ and  $t=\langle b\sb0,\dots, b\sb {n-1}\rangle$ from $S\sp2(\alpha,H)\sp{>n_{\a,\b}}$ satisfying the hypotheses of the Lemma and such that they falsify the thesis.
First, let us observe that for all $\gamma\leqo \beta$ such that $|\gamma | < \min s$, there are $k'\leq k$ and $n'\leq n$ such that 
$$
\gamma=\beta[a\sb0]\cdots [a\sb{k'-1}]=\beta[b\sb0]\cdots[b\sb{n'-1}].
$$
Indeed, if $|\gamma| < \min s < \min t$ (by Remark \ref{pairingfunctionproperty} about the choice of the pairing function) and $\gamma\leqo \beta$ then, by Proposition \ref{prop: goodnorm}, $\gamma$ has to occur as $\beta[a\sb0]\cdots [a\sb{k'-1}]$ and as $\beta[b\sb0]\cdots [b\sb{n'-1}]$ for some $k'\leq k$ and $n'\leq n$. 

So, let $\gamma\leqo\beta$ and $k'\leq k$ and $n' \leq n$ be such that 
$$
\gamma=\beta[a\sb0]\cdots [a\sb{k'-1}]=\beta[b\sb0]\cdots [b\sb{n'-1}].
$$
The numbers $k'$ and $n'$ are unique with respect to $\gamma$ and to the sequences $s$ and $t$ so we will denote them by $k\sb\gamma$ and $n\sb\gamma$ respectively. 
We also set $k\sb\beta=n\sb\beta=0$ as $\beta[\emptyset]=\beta$. 


Let us observe that for all $\gamma\leqo\beta$ such that $k\sb\gamma$ and $n\sb\gamma$ are defined and for all $\la\delta,e\ra<a\sb{k\sb\gamma}$, again by Lemma \ref{lem:from:alpha:to:beta} we have the following equivalences.
$$
M\sb\beta(\la\delta,e\ra,s)\mbox{ accepts} \ifff          M\sb\gamma(\la\delta,e\ra,\langle a\sb{k\sb\gamma},\dots, a\sb k \rangle)\mbox{ accepts}
$$
and
$$
M\sb\beta(\la\delta,e\ra,t)\mbox{ accepts}\ifff  M\sb\gamma(\la\delta,e\ra,\langle b\sb{n\sb\gamma},\dots, b\sb n\rangle)\mbox{ accepts}.
$$
For all $\delta \leqo\gamma$ such that $|\delta|< a\sb{k\sb\gamma}$ (by Remark \ref{pairingfunctionproperty} about the choice of the pairing function) the numbers $k\sb\delta$ and $n\sb\delta$ are also well defined. 
Indeed, if $\gamma=\beta[a\sb0]\cdots [a\sb{k'-1}]$ and $|\delta|\leq a\sb{k'}$ is such that $\delta\leqo\gamma$ then there exists $k''$ such that $k'\leq k''\leq k$ and $\delta=\beta[a\sb0]\cdots [a\sb{k''-1}]$. 
Indeed, since $|\delta|\leq a\sb{k'}$  we may again use Proposition \ref{prop: goodnorm}.

By the above observations we thus rule out the case in which $M\sb\beta(\la \delta,e \ra,t)$ accepts some $\la\delta,e\ra$ below $a\sb{k\sb\gamma}$ while $\delta$ does not occur as $\beta [a\sb0]\cdots [a\sb{k'-1}]$ for some $k'\leq k$.

For each $\gamma$ such that $\gamma \leqo\beta$ and such that $k\sb\gamma=k\sb\gamma(s)$ and $n\sb\gamma=n\sb\gamma(t)$ are defined let us define sets
\begin{eqnarray*}
Y\sp {\gamma}\sb{s}&=&
\set{\la\delta,e\ra<a\sb{k\sb\gamma}\colon M\sb\gamma(\la\delta,e\ra,\langle a\sb{k\sb\gamma},\dots,a\sb k\rangle)\mbox{ accepts}},\\
Y\sp{\gamma}\sb{t}&=&      \set{\la\delta,e\ra<b\sb{n\sb\gamma }\colon M\sb\gamma(\la\delta,e\ra,\langle b\sb{n\sb\gamma},\dots,b\sb n\rangle)\mbox{ accepts}}.
\end{eqnarray*}
The above sets are uniformly $\Delta^0_1$-definable with a parameter bounding computations of machines $M\sb\gamma (y,\langle b\sb{n\sb\gamma},\dots,b\sb n\rangle)$ on finite sets of inputs.

We claim that for each $\gamma \leqo \beta$ such that $k_\gamma \leq k$ the sets $Y\sp {\gamma}\sb{s}$ and $Y\sp{\gamma}\sb{t}$ coincide on all $\langle \delta, e\rangle < a_{k_\gamma}$.
Towards a contradiction, we define the set of indices $D$ on which the corresponding $Y\sp {\gamma}\sb{s}$ and $Y\sp{\gamma}\sb{t}$ differ:
\begin{multline*}
D=\{k'\leq k\colon \exists \gamma\leq a\sb{k'} 
(k'=k\sb\gamma  \land  \\
\exists \delta<a\sb{k\sb\gamma}\exists e<a\sb{k\sb\gamma}(\delta\leqo\gamma\land
\la\delta,e\ra<a\sb {k\sb\gamma} \land  \\
\neg(\la\delta,e\ra\in Y\sp{\gamma}\sb{s} \ifff \la\delta,e\ra\in Y\sp{\gamma}\sb{t}))\}.
\end{multline*}
The set $D$ is $\Delta^0_1$-definable with a sufficiently large parameter. Suppose by way of contradiction that $D$ is nonempty.

Let $k\sb0=k\sb{\gamma\sb0}$ be a maximal element of $D$ and let $n\sb0=n\sb{\gamma\sb0}$.
By the construction of the machine $M\sb{1}$,  $1\leo \gamma\sb0$ and $\gamma\sb0$ has to be a successor ordinal.
Otherwise, if $\gamma\sb0$ is limit then we would also have $k\sb0+1\in D$ contradicting the maximality of $k\sb0$. 
Thus, let $\gamma\sb0=\gamma\sb 1\pluso 1$ and let $\la\delta,e\ra<a\sb{k\sb0}$ be witnesses for $k\sb0\in D$, that is $\la\delta,e\ra<a\sb{k\sb0}$ and 
$$
\neg(\la\delta,e\ra\in Y\sp{\gamma\sb0}\sb{s} \ifff \la\delta,e\ra\in Y\sp{\gamma\sb0}\sb{t}),
$$
where $Y\sp{\gamma\sb 0}\sb{s}$ and $Y\sp{\gamma\sb 1}\sb{t}$ are defined as above. 

We claim that $\delta=\gamma\sb0$. 
Otherwise, we would have $k\sb{\delta}\in D$ for $k\sb\delta>k\sb{\gamma\sb0}$, as $k\sb\delta$ is well defined. Note that $\delta = \gamma_0$ implies $\langle \gamma_0, e\rangle < a_{k_0}$.

We distinguish two cases and derive a contradiction in both of them.

\smallskip

\textbf{Case 1:} $\langle\gamma\sb0,e\rangle\in Y\sp{\gamma\sb0}\sb{t}$. 
Then
$$
\set{e}\sp{Y\sp{\gamma\sb 1}\sb{t}}(0)\stops\sb{b\sb{n\sb0+1}}.
$$
We need to prove that 
$$
\set{e}\sp{Y\sp{\gamma\sb 1}\sb{s}}(0)\stops\sb{a\sb{k\sb0+1}}.
$$
Note that $\gamma\sb 1 < \gamma\sb 1\pluso 1$ by properties of the system of fundamental sequences and $\gamma\sb 1\pluso 1 = \gamma_0 <a\sb{k\sb0}$ as observed above. Moreover, $\langle b\sb{n\sb0+1},\dots, b\sb n\rangle$ is a $\gamma_1$-size sequence, and $a\sb{k\sb0} < a\sb{k\sb0+1} < b\sb{n\sb0+1}$ since $s<t$ by hypotheses of the Lemma. 
Therefore we can apply Lemma \ref{lem:bounding:computations} and conclude that
$$
\set{e}\sp {Y\sp{\gamma\sb 1}\sb{t}}(0)\stops\sb{a\sb{k\sb0+1}}.
$$
We argue that each query of type $\langle\delta,z\rangle$ asked during the computation of $\set{e}\sp {Y\sp{\gamma\sb 1}\sb{t}}(0)\stops\sb{a\sb{k\sb0+1}}$ is answered in the same way by the oracles $Y\sp{\gamma\sb 1}\sb{t}$ and $Y\sp{\gamma\sb 1}\sb{s}$. 
Indeed, let us consider $\langle\delta, z\rangle<a\sb{k\sb0+1}$. 
We may assume that $\delta\leqo\gamma\sb 1$ since otherwise $\langle\delta,z\rangle$ belongs neither to $Y\sp{\gamma\sb 1}\sb{t}$ nor to $Y\sp{\gamma\sb 1}\sb{s}$. 
Then $\delta<a\sb{k\sb0+1}$ and there are $k\sb\delta$ and $n\sb\delta$ such that 
$$
\delta=\gamma\sb 1[a\sb{k\sb0+1}]\cdots [a\sb{k\sb\delta-1}]=
\gamma\sb 1[b\sb{n\sb0+1}]\cdots [b\sb{n\sb\delta-1}].
$$
By the maximality of $k\sb0$ we get that $\langle\delta,z\rangle\in Y\sp{\gamma\sb 1}\sb{t}$ if and only if $\langle\delta, z\rangle\in Y\sp{\gamma\sb 1}\sb{s}$.
So, we proved that $\la \gamma\sb0,e\ra\in Y\sp{\gamma\sb0}\sb{s}$, contradicting the choice of $\langle\gamma\sb0,e\rangle$.
\smallskip

\textbf{Case 2:} $\langle\gamma\sb0,e\rangle\in Y\sp{\gamma\sb0}\sb{s}$, so
$\set{e}\sp {Y\sp{\gamma\sb 1}\sb{s}}(0)\stops\sb{a\sb{k\sb0+1}}$.
Then, we argue as in the previous case that 
$\set{e}\sp{ Y\sp{\gamma\sb 1}\sb{t}}(0)\stops\sb{a\sb{k\sb0+1}}$ and, therefore, $\la\gamma\sb0,e\ra \in Y\sp{\gamma\sb0}\sb{t}$, contradicting again the choice of~$\langle\gamma\sb 0,e\rangle$.
\end{proof}

Now, we can prove the full version of the lemma we need.

\begin{lemma}[$\RCA+\WO(\Lambda)$]\label{lem:eq:homo:wit}
Let $H\subseteq\N$ be an infinite $c\sb{\alpha \pluso3}$-\homo\ set. 
Then, for each $\beta \leqo \alpha$ and each $\beta$-size sequences $s, t$ in $S\sp2(\alpha,H)\sp{>n_{\a,\b}}$, it holds that
$$
{M\sb\beta(\langle\gamma,e\rangle,s)\mbox{ accepts}} \ifff
{M\sb\beta(\langle\gamma,e\rangle,t)\mbox{ accepts}},
$$
for all $\langle\gamma,e\rangle<\min(s,t)$.
\end{lemma}

\begin{proof}
The lemma follows easily from its simpler version, Lemma~\ref{lem:lem:eq:homo:wit}.
Indeed, if we take two $\beta$-size sequences $s,t$ in $S\sp2(\alpha,H) \sp{>n_{\a,\b}}$ then it is enough to take a third $\beta$-size sequence $r$ in $S(\alpha,H)\sp{>n_{\a,\b}}$ such that $\max(s,t)<\min r$. 
Now, both pairs $s, r$ and $t, r$ 
satisfy the assumptions of Lemma \ref{lem:lem:eq:homo:wit}. 
Then, for all $\la\gamma,e\ra<\min(s,t)$
we have
\begin{align*}
M\sb\beta(\langle\gamma,e\rangle,s)\mbox{ accepts}& \ifff M\sb\beta(\langle\gamma,e\rangle,r)\mbox{ accepts}\\
&\ifff M\sb\beta(\langle\gamma,e\rangle,t)\mbox{ accepts}.
\end{align*}
This gives us the thesis.
\end{proof}

\begin{lemma}[$\RCA+\WO(\Lambda)$]\label{lem:homo:wit:delta1}
Let $\alpha\in\Lambda$ and let $H \subseteq \N$ be a $c\sb{\alpha\pluso3}$-\homo\ set. 
Then, for each $\beta\leqo\alpha$ we have the following equality
\begin{multline*}
\set{y\colon \exists t \subseteq S\sp2(\alpha,H)\sp{>\max \set{y,n_{\a,\b}}}\ M\sb\beta(y,t)\mbox{ accepts}}= \\
\set{y\colon \forall t \subseteq S\sp2(\alpha,H)\sp{>\max \set{y,n_{\a,\b}}}\ M\sb\beta(y,t)\mbox{ accepts}}.
\end{multline*}
\end{lemma}

\begin{proof}
A direct consequence of Lemma \ref{lem:eq:homo:wit}.
\end{proof}

Now, we may define a set which will turn out to be the $\alpha$-th Turing jump.
Let us point out that we need to iterate the operation $S(\alpha,X)$ three times to obtain the desired set. 

\begin{definition}
Let $H\subseteq\N$ be an infinite $c\sb{\alpha \pluso3}$-\homo\ set where the coloring $c\sb{\alpha\pluso3}$ is relative to an oracle $A$ and let $\beta\leqo\alpha$.
We define a set $T(A,\alpha,H)$ as follows
$$
y\in T(A,\alpha,H) \ifff
\exists t \subseteq S\sp3(\alpha,H)\sp{>\max \set{y,n_{\a,\a}}}\ M\sb\alpha(y,t)\mbox{ accepts}.
$$ 
\end{definition}

We need the following technical lemma.
It allows us to move between $M\sb\alpha$ and $M\sb\beta$, for $\beta \leqo \alpha$ on inputs $\la \gamma,z \ra$ such that $\gamma\leqo\beta$.

\begin{lemma}[$\RCA+\WO(\a)$]\label{lem:T:beta:gamma}
Let $H \subseteq \N$ be an infinite $c\sb{\alpha\pluso3}$-\homo\ set, where $c\sb{\alpha\pluso3}$ has access to an oracle $A$ and let $\gamma \leqo \beta \leqo \alpha$.
Let $y=\la \gamma,z \ra$, then
\begin{multline*}
\exists t \subseteq S\sp3(\alpha,H)\sp{>\set{y,n_{\a,\b}}}M\sb\beta(y,t)\mbox{ accepts }
\ifff \\
\exists t' \subseteq S\sp3(\alpha,H)\sp{>\set{y,n_{\a,\a}}}M\sb\alpha(y,t')\mbox{ accepts }.
\end{multline*}
\end{lemma}

\begin{proof}
For the right to left direction let $t'= \langle t'\sb0,\dots,t'\sb n\rangle\subseteq  S\sp3(\alpha,H)\sp{>\set{y,n_{\a,\a}}}$ be such that  $M\sb\alpha(y,t')$ accepts. 
By Lemma~\ref{lem:eq:homo:wit}, we may choose $t'\sb0$ such that $|\beta| \leq y < t'\sb0$. 
Then, there exists an initial segment $\langle t'\sb0,\dots, t'\sb m\rangle$ of $t'$ such that $\alpha[t'\sb0]\cdots [t'\sb m]=\beta$.
Thus, the final segment $\langle t'\sb{m+1},\dots, t'\sb n\rangle$ of $t'$ is $\beta$-size. 
Moreover, by Lemma \ref{lem:from:alpha:to:beta}, $M\sb\beta(y,\langle t'\sb{m+1},\dots,t'\sb n\rangle)$ accepts.

For the other direction, let $t=\langle t\sb0,\dots,t\sb k\rangle\subseteq S\sp3 (\alpha, H)\sp{>\set{y,n_{\a,\b}}}$ be such that $M\sb\beta(y,t)$ accepts. 
By Lemma \ref{lem:homo:wit:delta1} take a sequence $\bar{t}\in S\sp 3(\alpha, H)\sp{> \set{y, n_{\a,\b}}}$ such that $|\alpha| < \min \bar{t}$ and $M\sb\beta(y,\bar{t})$ accepts.
Then, by Lemma \ref{lem:scattering:properties}, we can choose $\bar{s}\subseteq S\sp 2(\alpha, H)\sp{> \set{y, n_{\a,\a}}}$ such that $\alpha[\bar{s}]=\beta$ and $\bar{s} < \bar{t}$.
Then, Lemma \ref{lem:from:alpha:to:beta}, gives 
$$
M\sb\alpha(y,\bar{s} \conc \bar{t})\mbox{ accepts.}
$$
By Lemma~\ref{lem:homo:wit:delta1} we can choose any $\alpha$-size sequence from $S\sp2(\alpha,H)\sp{>\set{y,n_{\a,\a}}}$ to witness the above acceptance. 
In particular, we can choose from the subset $S\sp3(\alpha,H)\sp{>\set{y,n_{\a,\a}}}$ of $S\sp2(\alpha,H)\sp{>\set{y,n_{\a,\a}}}$.
\end{proof}

By Lemma~\ref{lem:homo:wit:delta1}, $T(A,\alpha,H)$ is $\Delta^0\sb1$-definable and therefore it is a set under $\RCA+\WO(\Lambda)$ (notice that $S\sp3(\alpha, H)\sp{>\set{y, n_{\a,\a}}}\subseteq S\sp2(\alpha, H)\sp{>\set{y, n_{\a,\b}}}$).
Thus, it does not matter what sequence $t$ we use in checking whether $y\in T(A,\alpha,H)$. 
Moreover, we have the following useful property. 
It states that while checking the membership in $T(A,\alpha,H)$ for a bounded number of elements (e.g.\ for a bounded oracle set) we may fix just one sequence $t$.

\begin{lemma}[$\RCA+\WO(\a)$]\label{lem:T:bounding}
Let $H\subseteq\N$ be an infinite $c\sb{\alpha \pluso3}$-\homo\ set, where $c\sb{\alpha\pluso3}$ has access to an oracle $A$ and let $\beta\leqo\alpha$. 
The following hold.
\begin{enumerate}
\item\label{3.20.1} Let $t$ be an $\a$-size sequence in $S\sp3(\alpha, H)\sp{> n_{\a,\a}}$.
For all $y<\min t$
$$
y\in T(A,\alpha,H) \ifff M\sb\alpha(y,t) \mbox{ accepts}.
$$
\item\label{3.20.2} Let $t$ be a $\beta$-size sequence in $S\sp3(\alpha, H)\sp{>n_{\a,\b}}$.
For all $\la \gamma,e \ra<\min t$ such that $\gamma\leqo\beta$ the following equivalence holds.
$$
\la \gamma,e \ra \in T(A,\alpha,H) \ifff M\sb\beta(\langle\gamma,e\rangle,t) \mbox{ accepts}.
$$
\end{enumerate}
\end{lemma}

\begin{proof}
The first part follows from the equivalence between $\Sigma\sp0\sb1$ and $\Pi\sp0\sb1$  definitions of $T(A,\alpha,H)$ following from Lemma \ref{lem:homo:wit:delta1}. 
To prove the second part let us fix a $\beta$-size sequence $t$ in $S\sp3(\alpha, H)^{> n_{\a,\b}}$. 
Then, by Lemma~\ref{lem:homo:wit:delta1}, for each $\langle\gamma,e\rangle<\min t$ with $\gamma\leqo\beta$, 
\begin{align*}
M\sb{\beta}(\langle\gamma,e\rangle,t)\mbox{ accepts} & \ifff \ \exists t' \subseteq S\sp 3(\alpha, H)\sp{> \set{\la \gamma, e\ra, n_{\a,\b}}}\, M\sb{\beta}(\la \gamma, e \ra,t')\mbox{ accepts.}
\end{align*}
Now, we use the equivalence from Lemma \ref{lem:T:beta:gamma} to get, for all $\la \gamma,e \ra < \min t$ with $\gamma \leqo \beta$, that
$$
M\sb\beta(\langle\gamma,e\rangle,t)\mbox{ accepts} \ifff \exists t \subseteq S\sp 3(\alpha,H)\sp{>\set{\langle\gamma,e\rangle,\alpha}}\, M\sb\alpha (\la \gamma, e\ra, t) \mbox{ accepts}.
$$
The last equivalence gives that, for each $\la \gamma, e \ra < \min t$ with $\gamma \leqo \beta$,
$$
M\sb\beta(\la \gamma,e \ra,t)\mbox{ accepts.}
\ifff \la\gamma,e\ra \in T(A,\alpha,H)  
$$
and this ends the proof.           
\end{proof}

Finally, we obtain the following corollary.

\begin{corollary}[$\RCA+\WO(\a)$]\label{cor:T:to:beta}
Let $H\subseteq\N$ be an infinite homogeneous set for $c\sb{\alpha \pluso3}$, where $c\sb{\alpha\pluso3}$ may use an oracle $A$, and $\beta\leqo\alpha$. 
For each $\gamma\leqo \beta$ and $e$
$$
\la\gamma,e\ra\in T(A,\alpha,H) \ifff \exists t\subseteq S\sp3(\alpha, H)\sp{> \set{\la \gamma,e \ra ,n_{\a,\b}}}\ M\sb\beta (\la \gamma,e \ra,t)\mbox{ accepts.}
$$
\end{corollary}

Now, we can prove that the defined above set $T(A,\alpha,H)$ satisfies the definition of the $\alpha$-th Turing jump.


\begin{theorem}[$\RCA+\WO(\a)$]\label{thm:jump:from:homo}
Let $H\subseteq\N$ be an infinite \homo\ set for $c\sb{\alpha\pluso3}$. 
The set $T(A,\alpha,H)$ satisfies the formula defining the $\alpha$-th Turing jump of~$A$.
\end{theorem}

\begin{proof}
It is obvious that a tuple $\la 0,z \ra \in T(A,\alpha,H)$ if and only if $z\in A$.

Now, we show that $T(A,\alpha,H)$ satisfies the recursive clause in the definition of Turing jump.
Recall that $\TJ(X,\alpha\pluso1)$, is the set 
$$\TJ(X,\alpha) \cup \set{\la \alpha \pluso 1, e \ra \colon \set{e}^{\TJ(X, \alpha)}(0)\downarrow }.$$

Let us take $y=\la \beta \pluso 1, z \ra$ for some $\beta\leo\alpha$. 
We have the following equivalences
\begin{align*}
\langle \beta \pluso 1, z \rangle \in T(A,\alpha,H)
& \ifff \exists t' \subseteq S^3(\alpha, H)^{> \{y,n_{\alpha,\alpha}\}}
   \ M_{\alpha}(y,t') \text{ accepts} \\
& \ifff \exists t \subseteq S^3(\alpha,H)^{> \{y,n_{\alpha,\beta \pluso 1}\}}
   \ M_{\beta \pluso 1}(y,t) \text{ accepts} \\
& \ifff \exists t \subseteq S^3(\alpha,H)^{> \{y,n_{\alpha,\beta \pluso 1}\}}
   \ \{z\}^{Y^*}(0)\!\downarrow_{\min t^-}, \\ 
& \qquad \text{where }
   Y^*=\{ w < \max t^- : M_\beta(w,t^-) \text{ accepts} \} \\
& \ifff \exists t \subseteq S(\alpha,H)^{> \{y,n_{\alpha,\beta \pluso 1}\}}
   \ \{z\}^{Y}(0)\!\downarrow_{\min t^-}, \\ 
& \qquad \text{where }
   Y=\left\{\langle\gamma,e\rangle < \min t :
   \gamma \leqo \beta \right. \\
& \qquad\qquad\left. \text{and } 
   \langle\gamma,e\rangle \in T(A,\alpha,H)\right\}.
\end{align*}
The first equivalence is just the definition of $T(A,\alpha,H)$. 
The second one is by Lemma \ref{lem:T:beta:gamma}. 
The third one is  by the definition of $M\sb{\beta+1}$. 
The fourth equivalence is based on the observation that if $\{z\}^{Y^*}(0)\stops_{\min(t^-)}$ with $Y^*=\set{w < \max t^- \colon M\sb\beta(w, t^-)\mbox{ accepts}}$ then also $\{z\}^{Y}(0)\stops_{\min(t^-)}$, where $Y=\set{w < \min t^- \colon M\sb\beta(w, t^-)\mbox{ accepts}}$.
The fifth equivalence is by \ref{3.20.2} of Lemma \ref{lem:T:bounding}, as we choose a sequence $t$ such that $\min t > |\beta \pluso 1| > |\beta|$ (recall that in Subsection \ref{ssec:systems} we assumed without loss of generality that the second inequality holds).
From the last formula we easily infer 
\begin{equation}\label{eq:T}
\{z\}\sp{Y}(0)\stops, \ \mbox{ where }Y=\set{\langle\gamma,e\rangle\in T(A,\alpha,H)\colon \gamma\leqo\beta}.
\end{equation}

For the other direction, let $y=\la\beta + 1,z\ra$ and let us assume the above formula~\eqref{eq:T}. 
Then, we infer that there exists $t\sb 0>\max(y,|\beta \pluso 1|)$ such that $\{z\}\sp{Y}(0)\stops\sb{t\sb 0}$.
Now we can compute, starting from $t\sb 0$, the whole $(\beta \pluso 1)$-size sequence $\langle t\sb0,\dots,t\sb n\rangle$ in $S\sp3(\alpha,H)$ such that $\langle\beta \pluso 1,z\rangle<t\sb 0$ and $\{z\}\sp{Y}(0)\stops\sb{t\sb 1}$ and, by \ref{3.20.2} of Lemma \ref{lem:T:bounding}, since $t^- = \langle t_1, \dots, t_n \rangle$ is $\beta$-size, the following holds: 
\begin{multline*}
Y\cap\set{0,\dots, t_1}=\\
\set{\langle\gamma,e\rangle\in T(A,\alpha,H)\colon \gamma\leqo\beta\mbox{ and } \langle\gamma,e\rangle<b\sb 1} =\\
\set{\langle\gamma,e\rangle<b\sb 1 \colon M\sb\beta(\langle\gamma,e\rangle,t^-)\mbox{ accepts}} .
\end{multline*}
This gives us that 
\begin{multline*}
\exists t \subseteq S\sp3(\alpha,H)\sp{>\set{y,n_{\a,\beta \pluso 1}}}\ 
\{z\}\sp{Y}(0)\stops\sb{ \min t},\\ \mbox{ where }  Y=\set{w\colon M\sb\beta(w, t)\mbox{ accepts}}. 
\end{multline*}

By Corollary~\ref{cor:T:to:beta} we get that the set $Y$ is $T(A,\alpha,H)$ restricted to elements $\langle \gamma,z\rangle<b\sb 1$, where $\gamma\leqo\beta$. 
Thus, 
$$
\la\beta\pluso1,z\ra\in T(A,\alpha,H).
$$

Let us observe that we do not need to consider the case for a limit ordinal $\lambda\leqo\alpha$. 
Indeed, the inductive condition for a jump for a limit ordinal $\lambda$ is just a sum of lower stages. 
This is satisfied by any set which is a candidate for being 
a $\lambda$ jump set.
Therefore, we showed that $T(A,\alpha,H)$ satisfies the recursive clause in the definition of the $\a$-th Turing jump. 
\end{proof}


\begin{corollary}\label{lowboundtj}
For all $0 < \alpha  < \Gamma_0$, $\RT^{!\omega^\alpha}_2$ implies $\Pi^0_{\omega^\alpha} \CA$ over $\RCA$.
\end{corollary}

\begin{proof}
By Lemma \ref{wellorder a} we have $\RT^{!\omega^\alpha}_2$ implies $\WO(\omega^\alpha)$ over $\RCA$. 
By the results of this section we have that $\RT^{!\omega^\alpha}$ implies $\forall X \exists Y (Y = X^{(\omega^\alpha)})$ 
over $\RCA + \WO(\omega^\alpha)$. 
\end{proof}

We observe that the proof presented in this section uses the principle $\RT^{!\alpha}_2$ only on $\N$ and does not need the generalized form on countable $X\subseteq\N$.

\subsection{Proof through well-ordering principles}

The second lower bound proof is based on a characterization of the systems $\Pi^0_{\omega^\alpha} \CA$ in terms of well-ordering preservation principles, due to Marcone and Montalbàn \cite{marconemontalban}.
They showed that for each $\alpha$ the system $\Pi^0_{\omega^\alpha} \CA$ is equivalent to the statement \lq\lq for each linear order $\X$ if $\X$ is a well-ordering then $\varphi_\a (\X)$ is a well-ordering\rq\rq, where $\varphi_\a (\X)$ is an operator, depending on $\a$, that maps linear orders to linear orders. 
We call the previous statement the {\em well-ordering principle for} $\varphi_\a(\X)$ and denote its standard formalization in second-order arithmetic by $\mathsf{WOP} (\X \to \varphi_\a(\X))$.
We now describe the operator $\varphi_\a(\X)$.
The general definition can be found in \cite{marconemontalban}.
Here we focus on the specific case we need.
For the rest of the section, let $\X$ be a linear ordering and let $\le_\X$ be its order relation.

\begin{definition}\label{wop def}
We define $\varphi_\a(\X)$ to be the set of formal terms defined as follows:
\begin{itemize}
    \item $0$ and $\varphi_\a(x)$ for $x \in \X$ belong to $\varphi_\a(\X)$ and are called constants,
    \item if $\mathtt{t}_1,\mathtt{t}_2 \in \varphi_\a(\X)$ then $\mathtt{t}_1+\mathtt{t}_2 \in \varphi_\a(\X)$,
    \item if $\mathtt{t} \in \varphi_\a(\X)$ and $\d < \a$ then $\varphi_\d(\mathtt{t}) \in \varphi_\a(\X)$.
\end{itemize}
\end{definition}

We denote terms by letters like $\mathtt{s}$, $\mathtt{t}$ and $\mathtt{u}$.

The normal form of a term and the order relation $\le_{\varphi_\a(\X)}$ on $\varphi_\a(\X)$ are defined simultaneously.
We say that a term $\mathtt{t} = \mathtt{t}_0 + \ldots + \mathtt{t}_n$ is in \textit{normal form} if either $\mathtt{t} = 0$ or $\mathtt{t}_0 \ge \ldots \ge \mathtt{t}_n$ and each $t_i$ is either a constant or of the form $\varphi_\d(\mathtt{s}_i)$ for $\d < \a$ and $\mathtt{s}_i$ a term in normal form with $\mathtt{s}_i \ne \varphi_{\d'}(\mathtt{s}_i')$ for $\d' > \d$.
To write a term in normal form we apply the following rules:
\begin{itemize}
    \item $+$ is associative and $0$ is the neutral element,
    \item if $\varphi_{\d'}(\mathtt{s}) <_{\varphi_\a(\X)} \varphi_\d(\mathtt{u})$ then $\varphi_{\d'}(\mathtt{s}) + \varphi_\d(\mathtt{u}) = \varphi_\d(\mathtt{u})$,
    \item if $\d'>\d$ then $\varphi_\d (\varphi_{\d'}(\mathtt{s})) = \varphi_{\d'}(\mathtt{s})$,
    \item if $\d < \a$ then $\varphi_\d(\varphi_\a(\mathtt{s})) = \varphi_\a(\mathtt{s})$.
\end{itemize}

Given $\mathtt{t} = \mathtt{t}_0+\ldots+\mathtt{t}_n$ and $\mathtt{s}=\mathtt{s}_0+\ldots+\mathtt{s}_m$ written in normal form, we say that $\mathtt{t} \le \mathtt{s}$ if and only if one of the following applies:
\begin{itemize}
    \item $\mathtt{t}=0$,
    \item $\mathtt{t} = \varphi_\a(x)$ and for some $y \ge_\X x$, $\varphi_\a(y)$ occurs in $\mathtt{s}$,
    \item $\mathtt{t} = \varphi_\d(\mathtt{t}')$, $\mathtt{s}_0 = \varphi_{\d'}(\mathtt{s}')$ and 
    $\begin{cases}
    \d < \d' \text{ and } \mathtt{t}' \le_{\varphi_\a(\X)} \mathtt{s}_0 \text{ or } \\
    \d = \d' \text{ and } \mathtt{t}' \le_{\varphi_\a(\X)} \mathtt{s}' \text{ or } \\
    \d > \d' \text{ and } \mathtt{t} \le_{\varphi_\a(\X)} \mathtt{s}',
    \end{cases}$
    \item $n > 0$ and $\mathtt{t}_0 <_{\varphi_\a(\X)} \mathtt{s}_0$,
    \item $n > 0$, $\mathtt{t}_0 = \mathtt{s}_0$, $m > 0$ and $\mathtt{t}_1 + \ldots +\mathtt{t}_n \le_{\varphi_\a(\X)} \mathtt{s}_1 + \ldots \mathtt{s}_m$.
\end{itemize}

For $\mathtt{s}, \mathtt{t}\in \varphi_\a(\X)$ we may write $\mathtt{t} < \mathtt{s}$ for $\mathtt{t} <_{\varphi_\a(\X)} \mathtt{s}$ for ease of readability.

The following result is crucial for our lower bound proof of $\RT^{!\omega^\alpha}$. 

\begin{theorem}\label{wop mm11}\cite[Theorem 6.16]{marconemontalban}
Let $\a < \G_0$.
Then $\Pi^0_{\om^\a} \CA$ and $ \mathsf{WOP} (\X \to \varphi_\a(\X))$ are equivalent over $\RCA$.
\end{theorem}

Our second lower bound proof for $\a$-size Ramsey's theorems is based on the characterization of $\Pi^0_{\omega^\a} \CA$ from Theorem \ref{wop mm11}. 
We thus aim to prove over  $\RCA$ that $\RT^{!\omega^\alpha}_4$ implies $\mathsf{WOP} (\X \to \varphi_\a(\X))$. 
More precisely we show, under the assumption of $\RT^{!\omega^\alpha}_4$, that for each linear order $\X$, if there exists an infinite descending sequence in $\varphi_\alpha (\X)$, then there exists an infinite descending sequence in $\X$.

This strategy has been applied 
in \cite{carlucci2025} to obtain lower bounds for $\RT^n_k$ and $\RT^{! \omega}_2$. The details of the argument are lifted from the finite to the infinite domain from the lower bound proved on an extension of the Paris-Harrington principle in the recent \cite{marcone2025}. 
The situation is analogous to \cite{carlucci2025}, where the beautiful combinatorial proof of the Paris-Harrington theorem by Loebl and Ne\v{s}etril \cite{Loe-Nes:92} is lifted from the finite to the infinite domain to obtain an implication from $\RT^{! \om}_2$ to the well-ordering principle for the operator $\X \to \omega^\X$. 
The latter is equivalent to $\ACA$ by results of Girard and Hirst (see \cite{marconemontalban}).

The proof of $\RT_4^{!\om^\a} \vdash \Pi^0_{\omega^\a} \CA$ is by external induction on $\a<\G_0$.
The base case $\a=1$ is \cite[Theorem 3.6]{carlucci2014}, recalling that $\Pi^0_{\omega} \CA$ is $\ACA^+$. Thus, for the rest of this section we fix a positive $\a < \G_0$ and we suppose that for all $0 < \b < \a$, $\RT^{!\om^\b}_4 \vdash \Pi^0_{\omega^\b} \CA$.
Then we are left to prove $\RT_4^{!\om^\a} \vdash \Pi^0_{\omega^\a} \CA$.
Notice that the largeness notions involved are only for ordinals $\le \om^\a$.
Therefore, it suffices to consider the system of fundamental sequences on $\om^\a$ induced by the system of Definition \ref{system of fund seq}.


Our goal is to show that starting from an infinite strictly descending sequence in $\varphi_\a(\X)$ we can construct an instance $\overline{c} \colon [M]^{! \om^\a} \to 4$ such that every infinite homogeneous set computes an infinite strictly descending sequence in $\X$.

We now adapt the machinery of the peeling functions of \cite{marcone2025} to finite sequences of terms in $\varphi_\a(\X)$.
For the rest of the section, terms refer to elements of $\varphi_\a(\X)$ written in normal form according to the current subsection.

\begin{definition}\label{overline def}
Let $\mathtt{t},\mathtt{s} \in \varphi_\a (\X)$.
If $\mathtt{t} > \mathtt{s}$ let $i$ be the least index such that $\mathtt{t}_i > \mathtt{s}_i$.
If $\mathtt{t}_i = \varphi_0(\mathtt{t}')$ then $\overline{\mathtt{t},\mathtt{s}} = \mathtt{t}'$, otherwise $\overline{\mathtt{t},\mathtt{s}} = \mathtt{t}_i$.
If $\mathtt{t} \le \mathtt{s}$ then $\overline{\mathtt{t},\mathtt{s}} = 0$.
\end{definition}

Notice that the overline function is just a syntactic check, so it is computable.
		
Next, we introduce the so called \textit{peeling functions}, which are length-preserving functions on finite sequences of elements of $\varphi_\a(\X)$.
They can be seen as generalizations of functions used in \cite{Loe-Nes:92,carlucci2025}.
We denote finite sequences in $\varphi_\a(\X)$ with upper case letters of the form $\mathtt{A}$ and $\mathtt{B}$ from the beginning of the alphabet.
We also introduce the concept of the {\em collection of subterms} of a $\mathtt{t} \in \varphi_\a(\X)$.
To do that we use the notion of multiset, i.e.\ a set in which elements can occur multiple (though, finitely many) times.
Recall that if $A$ and $B$ are multisets then $A+B$ is the multiset where each element has multiplicity the sum of its multiplicities in $A$ and $B$.

\begin{definition}
We recursively define the multiset $\Sub (\mathtt{t})$ of subterms of a term $\mathtt{t}$ as follows:
\begin{itemize}
    \item if $\mathtt{t} = 0$, let $\Sub(\mathtt{t}) = \{0\}$,
    \item if $\mathtt{t} = \varphi_\a(x)$ with $x \in \X$ (a non zero constant), let $\Sub(\mathtt{t}) = \{\mathtt{t},x\}$,
    \item if $\mathtt{t} = \varphi_\d (\mathtt{t}')$ with $\d < \a$, let $\Sub(\mathtt{t}) = \{\mathtt{t}\} + \Sub(\mathtt{t}')$,
    \item if $\mathtt{t} = \mathtt{t}_0 + \ldots + \mathtt{t}_n$ with $n > 0$, let $\Sub(\mathtt{t}) = \{\mathtt{t}\} + \Sub(\mathtt{t}_0) + \ldots + \Sub(\mathtt{t}_n)$.
\end{itemize}
\end{definition}

Notice that $\Sub(\mathtt{t})$ is finite and depends only on the syntactic form of $\mathtt{t}$, so it is computable.

\begin{definition}\label{peeling def}
We define functions $\pbar_\d : \varphi_\a(\X)^{<\om} \to \varphi_\a(\X)^{<\om}$ by induction on $\d \le \a$ as follows.
Let $\pbar_0$ be the identity function, and let
$$\pbar_1 (\mathtt{t}_0, \mathtt{t}_1, \ldots, \mathtt{t}_\ell) = (\overline{\mathtt{t}_0,\mathtt{t}_1}, \overline{\mathtt{t}_1,\mathtt{t}_2}, \ldots, \overline{\mathtt{t}_\ell,0}).$$
For ordinals of the form $\rho + \om^\d$ with $\rho \ggeq \om^\d$, we let
$$\pbar_{\rho + \om^\d} = \pbar_{\om^\d} \circ \pbar_\rho.$$
On infinite ordinals of the form $\om^\d$, we first define
$$\pbar_{< \om^\d}(\mathtt{A}) = \lim_{\rho \to \om^\d} \pbar_\rho (\mathtt{A}).$$
We prove in Lemma \ref{basic peeling} below that for each ordinal $\nu \le \rho$ and each $i < |\mathtt{A}|$, $\pbar_\rho (\mathtt{A}) \in \Sub (\pbar_\nu (\mathtt{A}))$.
Therefore, $\pbar_\rho(\mathtt{A})$ is coordinate-wise non-increasing (with respect to the order relation we defined for terms) as a function of $\rho$, and since there are only finitely many subterms of a fixed term, the limit above exists.
Furthermore, we see in Lemma \ref{basic peeling} that each term in the tuple $\pbar_{< \om^\d} (\mathtt{A})$ is either an element of $\X$, a constant term or of the form $\varphi_\b(\mathtt{s})$ for some $\b \ge \d$ and $\mathtt{s}$ in normal form.
Then we define $\pbar_{\om^\d} (\mathtt{A})$ by peeling off one application of $\varphi_\d$ from each non-zero entry in $\pbar_{<\om^\d} (\mathtt{A})$: if the entry is of the form $\varphi_\d(\mathtt{s})$, then we get $\mathtt{s}$; if it is a non-zero constant or is of the form $\varphi_\b(\mathtt{s})$ for some $\b > \d$, then we get the same entry. 

We write $p_\d (\mathtt{A})$ (without the bar) for the first element of $\pbar_\d (\mathtt{A})$.
\end{definition}

Notice that to be consistent with the terminology we may actually say that $\pbar_1$ peels off one application of $\varphi_0$.
This is appropriate since $\pbar_1 = \pbar_{\om^0}$.

Also notice that the peeling functions are computable: this is because $\pbar_1$ is basically the overline function which is computable, while for indecomposable ordinals we know by Lemma \ref{basic peeling} that after some finite number of steps (which depends on the input $\mathtt{A}$), $\pbar_{< \om^\d}(\mathtt{A})$ stabilizes and then we peel off $\varphi_\d$.
Again we can syntactically check when $\pbar_{< \om^\d}(\mathtt{A})$ has stabilized.
		
\begin{lemma}[$\RCA + \RT^{! \om^\a}_4$]\label{basic peeling}
The following properties of the peeling functions hold.
\begin{enumerate}
	\item \label{31} For every $\rho \le \om^\a$ and $\nu < \rho$ we have that $\pbar_\rho (\mathtt{A}) (i) \in \Sub(\pbar_\nu (\mathtt{A}) (i))$ for each $i < |\mathtt{A}|$.
	\item \label{32} Each entry of $\pbar_{<\om^\d}(\mathtt{A})$ is either a constant or of the form $\varphi_\b(\mathtt{s})$ for some $\b \ge \d$ and $\mathtt{s}$ in normal form with $\mathtt{s} \ne \varphi_{\b'}(\mathtt{s}')$ for $\b' > \b$.
	\item \label{33} For each $i < |\mathtt{A}|$, $\pbar_{\om^\a}(\mathtt{A})(i) \in \X$.
\end{enumerate}
\end{lemma}

\begin{proof}
The proof is essentially the same as the proof of \cite[Lemma 6.5]{marcone2025}: the argument is given there for tuples of ordinals, but it can be easily adapted to tuples of terms.
Here we observe that it goes through in $\RCA + \RT^{! \om^\a}_4$.

Statements $\ref{31}$ and $\ref{32}$ are proved using arithmetical induction over the well-order $\om^\a$.
By Corollary \ref{induction} this induction is available in $\RCA + \RT^{! \om^\a}_4$.
\ref{33} is an immediate consequence of $\ref{32}$.
\end{proof}

We observe that no element of $\X$ is a term and so cannot be an argument of a peeling function, but it may happen that for some $\mathtt{A}$, some $i < |\mathtt{A}|$ and some $\d$, $\pbar_\d(\mathtt{A})(i) \in \X$.
However, this happens if and only if $\d = \om^\a$.

An important observation about the peeling functions is that each entry of $\pbar_\d(\mathtt{A})$ does not depend on the previous entries of the finite sequence $\mathtt{A}$.
In other words $\pbar_\d (\mathtt{A}^-) = (\pbar_\d (\mathtt{A}))^-$ and $\pbar_\d (\mathtt{A}^{-n}) = (\pbar_\d (\mathtt{A}))^{-n}$ where $\mathtt{A}^-$ denotes the finite sequence $\mathtt{A}$ without its first element, and $\mathtt{A}^{-(n+1)} = (\mathtt{A}^{-n})^-$.

\begin{definition}
Let $\mathtt{A}$ be a finite tuple of terms.
If $p_{\om^\a}(\mathtt{A}) \le_\X p_{\om^\a}(\mathtt{A}^-)$, let $\z_\mathtt{A} \le \om^\a$ be the least ordinal $\z$ such that $p_\z(\mathtt{A}) \le p_\z(\mathtt{A}^-)$.
If $p_{\om^\a} (\mathtt{A}) >_\X p_{\om^\a} (\mathtt{A}^-)$ then $\z_\mathtt{A}$ does not exist.
\end{definition}

As $p_\z(\mathtt{A}^-)$ is the second entry of $\pbar_\z(\mathtt{A})$, we have $p_{\z_\mathtt{A} + 1} (\mathtt{A}) = 0$.
If $\z_\mathtt{A}$ does not exist, then we have $p_{\om^\a} (\mathtt{A}) \in \X$ by Lemma \ref{basic peeling}.
		
\begin{definition}\label{def: coloring on ord}
Let $c$ be the following $4$-coloring of finite tuples of terms:
\begin{itemize}
	\item if $\z_\mathtt{A}$ does not exist, let $c(\mathtt{A}) = 0$,
	\item if $\z_\mathtt{A}$ and $\z_{\mathtt{A}^-}$ both exist and $\z_\mathtt{A} > \z_{\mathtt{A}^-}$, let $c(\mathtt{A}) = 1$,
	\item if $\z_\mathtt{A}$ and $\z_{\mathtt{A}^-}$ both exist and $\z_\mathtt{A}=\z_{\mathtt{A}^-}$, let $c(\mathtt{A}) = 2$,
	\item if $\z_\mathtt{A}$ exists and either $\z_{\mathtt{A}^-}$ does not or $\z_\mathtt{A} < \z_{\mathtt{A}^-}$, let $c(\mathtt{A}) = 3$.
\end{itemize}
\end{definition}

The coloring $c$ is computable since for each tuple $\mathtt{A}$ the ordinal $\z_\mathtt{A}$ is.

We define a function that assigns to each term $\mathtt{t} \in \varphi_\a(\X)$, a finite set of ordinals $\S(\mathtt{t})$ that contains all the ordinals which may be equal to $\z_\mathtt{A}$ for some tuple of terms $\mathtt{A}$ with $\mathtt{A}(0) = \mathtt{t}$.
		
\begin{definition}
We recursively define the set of ordinals $\S(\mathtt{t})$:
\begin{itemize}
	\item if $\mathtt{t}=0$, let $\S(\mathtt{t}) = \{0\}$,
    \item if $\mathtt{t}=\varphi_\a(x)$ for $x \in \X$ (i.e.\ a non zero constant), let $\S(\mathtt{t}) = \{0,1,\om^\a\}$,
    \item if $\mathtt{t} = \varphi_\d (0)$ for $\d < \a$, let $\S(\mathtt{t}) = \{0, 1, \om^\d\}$,
	\item if $\mathtt{t} = \varphi_\d (\mathtt{s})$ with $\d < \a$ and $\mathtt{s}$ in normal form, let $\S(\mathtt{t}) = \{0, 1\} \cup \{\om^\d + \xi: \xi \in \S(\mathtt{s}) \wedge \xi > 0\}$,
	\item if $\mathtt{t} = \mathtt{t}_0 + \ldots + \mathtt{t}_n$ with $n > 0$ written in normal form, let $\S(\mathtt{t}) = \{0,1\} \cup \S(\mathtt{t}_0) \cup \ldots \cup \S(\mathtt{t}_n)$.
\end{itemize}
\end{definition}

Notice that for each term $\mathtt{t}$ the corresponding $\S(\mathtt{t})$ is a finite set of ordinals $\le \om^\a$ and only depends on the syntactic form of $\mathtt{t}$, so it is computable.

\begin{lemma}[$\RCA + \RT^{! \om^\a}_4$]\label{lemma: Stau}
Let $\mathtt{t} \in \varphi_\a (\X)$.
Then for each tuple of terms $\mathtt{A}$ with $\mathtt{A}(0) = \mathtt{t}$, if $\z_\mathtt{A}$ exists then $\z_\mathtt{A} \in \S(\mathtt{t})$.
\end{lemma}

\begin{proof}
As Lemma $\ref{basic peeling}$, the proof is essentially an adaptation to terms of the argument for ordinals in \cite[Lemma 6.8]{marcone2025}.
Here we show that it goes through in $\RCA + \RT^{! \om^\a}_4$.

Fix a recursive enumeration of the terms such that, for each term $\mathtt{t}$, every element of $\Sub(\mathtt{t}) \setminus \{\mathtt{t}\}$ appears earlier in the list.  
We then argue by $\Pi^0_1$-induction along the resulting well-order of order-type $\omega$: indeed, the statement for a given term $\mathtt{t}$ depends only on the inductive assumptions for the terms in $\Sub(\mathtt{t})$.
This induction is available in $\RCA$.
Since we are also using basic properties of the peeling functions proved in Lemma \ref{basic peeling}, the proof can be carried out in $\RCA + \RT^{! \om^\a}_4$.
\end{proof}

We defined the norm $|\d|$ of an ordinal $\d < \G_0$ in Definition \ref{normdefinition}. 
We now extend the norm function to terms.

\begin{definition}\label{doublenorm}
For each term $\mathtt{t} \in \varphi_\a(\X)$ written in normal form, let $|\mathtt{t}|$ be $1$ plus the maximum of:
\begin{itemize}
    \item the cardinality of the multiset $\Sub (\mathtt{t})$,
	\item the $|\cdot|$ norm of all the ordinals in $\Sub (\mathtt{t})$,
	\item the $|\cdot|$ norm of all the ordinals in $\S(\mathtt{t})$.
\end{itemize}
\end{definition}

Notice that $|\mathtt{t}|$ is computable since the sets $\Sub (\mathtt{t})$ and $\S (\mathtt{t})$ are finite and computable and the $|\cdot|$ norm of an ordinal is computable.

In the coloring $c$ of Definition \ref{def: coloring on ord}, what matters is the behavior of the peeling function on the first entry of the tuple.
We show that we can foresee when $p_{ < \om^\d} (\mathtt{A})$ stabilizes and that it depends only on $\d$ and on $|\mathtt{A}(0)|$.

\begin{lemma}[$\RCA + \RT^{! \om^\a}_4$]\label{convergence lemma}
Let $0 < \d \le \a$ and let $\mathtt{A} = (\mathtt{t}_0, \ldots, \mathtt{t}_m)$ be a sequence of terms.
Then, 
$$p_{< \om^\d} (\mathtt{A}) = p_{\om^{\d[|\mathtt{t}_0|]} \cdot |\mathtt{t}_0|}(\mathtt{A}).$$ 
\end{lemma}

\begin{proof}
As before, the proof is the version for terms of \cite[Lemma 6.11]{marcone2025}.
The argument needs the first two properties of the norm of terms of Definition \ref{doublenorm}.
It also uses Lemma \ref{basic peeling}, arithmetical induction over $\om^\a$ and Proposition \ref{prop: goodnorm}.
These are all available in $\RCA + \RT^{! \om^\a}_4$ by Remark \ref{technical lemmas}.
\end{proof}

Let $\sigma \colon \N \to \varphi_\a(\X)$ be an infinite descending sequence.
Our strategy is to use $\sigma$ to translate the coloring $c$ to a coloring of $\om^\a$-size sets of numbers.
First of all, we need to make $\sigma$ more sparse to meet some technical requirements.

\begin{definition}\label{infinite set M}
Let $M$ be the infinite computable set defined as follows: $M(0) = 0$ and for each $i > 0$, $M(i) = |\sigma(M(i-1))| + 3$.
\end{definition}

Define
$$\tau \colon M^- \to \varphi_\a (\X) \quad \text{ by } \quad \tau(M(i)) = \sigma(M(i-1)).$$
Notice that $\tau$ is an infinite subsequence of $\sigma$ computable from $\sigma$.
Moreover, by definition of $M$, 
$$|\tau(M(i))| + 2 < M(i).$$		
		
Now we have all the required machinery.
We define a coloring on $[M^-]^{! \ge \om^\a}$ (which denotes the set of $\om^\a$-large subsets of $M^-$)
\begin{align*}
    \overline{c} \colon & \,\, [M^-]^{! \ge \om^\a} \to 4	\\
    & \,\, u \mapsto c(\tau(u))
\end{align*}
where $c$ is the coloring from Definition \ref{def: coloring on ord} and where by $\tau(u_0, \ldots, u_m)$ we mean $(\tau(u_0), \ldots, \tau(u_m))$.
Notice that $\overline{c}$ is computable from $c$, $M$ and $\tau$.
Similarly, we define $\pbar_\d$ and $\z$ on $u \subset M$ by  $\pbar_\d (u) = \pbar_\d(\tau(u))$ and $\z_u = \z_{\tau(u)}$.
		
We now show that in the definition of $\overline{c}$, only the $\om^\a$-size prefix matters, i.e.\ if $s \subseteq M^-$ and $u \sqsubseteq s$ is the $\om^\a$-size initial segment of $s$, then $\overline{c}(s) = \overline{c}(u)$.
		
\begin{lemma}[$\RCA + \RT^{! \om^\a}_4$]\label{lem: prefix coloring}
For each $\nu \le \om^\a$ and sets $u \sqsubseteq s \subset M^-$, if $u$ is $(1 + \nu)$-large then
$$p_\nu (u) = p_\nu (s).$$
\end{lemma}
		
\begin{proof}
The proof given in \cite[Lemma 6.13]{marcone2025} already works for terms instead of ordinals.
It uses Lemma \ref{convergence lemma} and arithmetical induction over $\om^\a$ and so by Corollary \ref{induction} can be carried out in $\RCA + \RT^{! \om^\a}_4$.
\end{proof}

When $\nu = \om^\a$ the above Lemma shows that the coloring $\overline{c}$ defined above can actually be regarded as a coloring of the $\om^\a$-size subsets of $M$ (since $1+\om^\a = \om^\a$).
This is because in the coloring $\overline{c}$ of a set $u$ we only care about the peeling functions of index $\om^\a$ or index $\nu \in \S(\tau(\min u))$.
By definition of the set $M$ (which contains $u$) each element of $u$ is larger than $|\tau(\min u)|$ and by Definition \ref{doublenorm} of the norm on terms, it is also larger than $|\nu|$ for each $\nu \in \S(\tau(\min u))$.
This means that if $u$ is $\om^\a$-large, by Proposition \ref{prop: goodnorm} (which is provable in $\RCA + \RT^{! \om^\a}_4$ by Remark \ref{technical lemmas}) since $\om^\a \ge 1+\nu$, $u$ must also be $(1+\nu)$-large.
Therefore Lemma \ref{lem: prefix coloring} yields that also $p_\nu$ only depends on the $(1+\nu)$-size initial segment of its input.

We are now ready to give the second proof of the lower bound of the Main Theorem \ref{thm:main}.
The proof can be carried out in $\RCA + \RT^{! \om^\a}_4$.


\begin{theorem}\label{th:homog sets}
Every infinite homogeneous set $H\subseteq M^-$ for $\overline{c}$ has color $0$ and computes an infinite descending sequence in $\X$.
\end{theorem}

\begin{proof}
Let $H\subseteq M^-$ be an infinite homogeneous set for $\overline{c}$.
For each $i\in\N$ let $s_i$ be the $\om^\a$-size initial segment of the infinite set $H \setminus \{0, \ldots, H(i)-1\}$ (so that $\min s_i = H(i)$).
By homogeneity, we have that $\overline{c} (s_i)$ has the same color for all $i \in \N$.
We distinguish four cases based on the color of the homogeneous set $H$.
We start by showing that if $H$ has color 0, then it computes an infinite descending sequence in $\X$.
If $H$ has a color different from 0 then we prove that it cannot be infinite.

\medskip

\textbf{Case 1:} $H$ has color 0.
By Lemma \ref{basic peeling}, for each $i\in\N$, $p_{\om^\a}(s_i) \in \X$.
We define a map
\begin{align*}
    \sigma' \colon & \,\, H \to \X \\
    & \,\, H(i) \mapsto p_{\om^\a}(s_i).  
\end{align*}
By Definition \ref{def: coloring on ord}, for each $i$, we have that $\z_{s_i}$ does not exists and hence $p_{\om^\a} (s_i) >_\X p_{\om^\a}(s_{i+1})$.
Therefore $\sigma'$ is an infinite strictly decreasing sequence in $\X$ computable from $H$ and $\overline{c}$ and so it is computable from $H$ and the strictly decreasing sequence $\sigma$ in $\varphi_\a(\X)$ we fixed before.

\medskip

\textbf{Case 2:} $H$ has color 1.
We have that
$$\z_{s_0} > \z_{s_1} > \cdots > \z_{s_n} > \cdots$$
and all these ordinals belong to $\om^\a + 1$.
This contradicts Lemma \ref{wellorder a} which states that $\om^\a$, and hence $\om^\a+1$, is a well-order.

\medskip

\textbf{Case 3:} $H$ has color 2 or color 3.
In this case either
$$\z_{s_0} = \z_{s_1} = \cdots = \z_{s_n} = \cdots,$$
or
$$\z_{s_0} < \z_{s_1} < \cdots < \z_{s_n} < \cdots$$
holds. 
We claim that the terms $p_{\z_{s_i}} (s_i)$ originate from distinct occurrences of the elements of the multiset $\Sub (\tau(H(0)))$, contradicting the fact that the set $H$ is infinite.
The proof is essentially an adaptation to our setting of the argument for cases 3 and 4 of \cite[Theorem 6.1]{marcone2025}.
\end{proof}

\begin{corollary}[$\RCA$]
$\RT_4^{!\om^\a} \vdash \Pi^0_{\om^\a} \CA$.
\end{corollary}

\begin{proof}
Let $\X$ be a linear ordering and let $\sigma$ be an infinite descending sequence in $\varphi_\a (\X)$.
Let $\overline{c}$ be the coloring associated to $\sigma$, $\a$, $\X$ and defined starting from the computable coloring $c$ of Definition \ref{def: coloring on ord}.
By $\RT_4^{!\om^\a}$, such coloring has an infinite homogeneous set $H$ of color 0.
By Theorem \ref{th:homog sets}, $H$ computes an infinite descending sequence in $\X$.
\end{proof}


%
%
%
%

\section{Upper bound}\label{sec:upper}

The upper bound is proved by induction on $\a < \G_0$.
The base case $\RT^{! \om}_k$ can be found in \cite[Theorem 3.7]{carlucci2014}.
The next lemma deals with the other basic case $\RT^{!1 \uplus \om}_k$: we isolate this case to highlight the strategy, which will be analogous for the induction step in the main result.

\begin{lemma}\label{upluslemma}
For each infinite $X \subseteq \N$, $k \in \N$, and coloring $c \colon [X]^{! 1 \uplus \omega} \to k$, $\TJ(c \oplus X,\om)$ computes an infinite homogeneous set for $c$.
\end{lemma}

\begin{proof}
We aim to define a strictly increasing sequence $(h_i)_{i \in \N}$ of elements of $X$ and a sequence $(H_i)_{i \in \N}$ of infinite subsets of $X$ with the following properties for each $i$:
\begin{enumerate}
    \item $H_{i+1} \subset H_i$;
    \item $h_i < \min H_i$;
    \item $H_i$ is computable from $\TJ(c \oplus X,m)$ for some $m \in \N$;
    \item for each $s,t \in [H_i]^{h_i+1}$ $c(\langle h_i \rangle \conc s) = c(\langle h_i \rangle \conc t)$.
\end{enumerate}

Initialize the construction by leaving $h_{-1}$ undefined and setting $H_{-1} = X$.
Suppose we have defined the sequences up to $h_{i-1}$ and $H_{i-1}$.
Let $h_i = \min H_{i-1}$ and consider the $c \oplus H_{i-1}$ computable coloring
\begin{align*}
    f_{h_i} \colon & [H_{i-1} \setminus \{h_i\}]^{h_i+1} \to k \\
     & t \mapsto c(\langle h_i \rangle \conc t)
\end{align*}
Then by \cite{jockusch} $\TJ(f_{h_i} \oplus H_{i-1},h_i+2)$ computes an infinite homogeneous set $H_i$ for $f_{h_i}$.
By inductive hypothesis of the construction we know that $H_i$ is computable in some finite jump of $c \oplus X$ and so the sequences up to $i$ satisfy the requirements.

Let $Z = \{h_i : i \in \N \}$ and let $f \colon Z \to k$ be defined as $f(h_i) = f_{h_i} (t)$ for some $t \in [H_i]^{h_i+1}$: by definition of $H_i$ the coloring $f$ is well defined and does not depend on the choice of $t \subseteq H_i$.
Moreover $f$ and $Z$ are computable from $\TJ(c \oplus X,\om)$ since this set uniformly computes the sequences $(h_i)_{i \in \N}$ and $(H_i)_{i \in \N}$.
Since $f$ is an instance of $\RT^1_k$ which admits computable solutions, we get that $\TJ(c \oplus X,\om)$ computes an infinite homogeneous set $H$ for $f$.
It is immediate to verify that $H$ is homogeneous for $c$ too.
\end{proof}

\begin{theorem}\label{upper general case}
Let $\a < \G_0$, $c \colon [X]^{! \a} \to k$ be an instance of $\RT^{! \a}_k$ and $d \colon [Y]^{! 1 \uplus \a} \to k$ be an instance of $\RT^{! 1 \uplus \a}_k$.
Then $\TJ(c \oplus X,\a+1)$ computes a solution to $c$ and $\TJ(d \oplus Y,\a+1)$ computes a solution to $d$.
\end{theorem}

\begin{proof}
The proofs are by simultaneous induction.
For both statements we consider separately the cases $\a$ decomposable and $\a$ indecomposable.

Notice that for $\a \le \om$ the statements have already been proved: for each $n \in \N$ the result for $\RT^n_k$ was proved in \cite{jockusch}, while for $\RT^{! \omega}_k$ it was proved in \cite{carlucci2014}.
Their proof is for the case $k=2$ but as we noticed at the end of Subsection \ref{ssec:ramsey theory}, the statements for the same $\a$ and different number of colors are equivalent over $\RCA$. 
For $\RT^{! 1 \uplus \omega}_k$ it was proved in Lemma \ref{upluslemma} above.
Therefore for the rest of the proof we fix $\alpha > \omega$ and suppose that both statements are true for every $\beta < \alpha$.

\medskip

\textbf{Case 1:} Suppose that $\alpha$ is indecomposable and let $c \colon [X]^{! \alpha} \to k$.
We perform the construction of a strictly increasing sequence $(h_i)_{i \in \N}$ of elements of $X$ and a sequence $(H_i)_{i \in \N}$ of infinite subsets of $X$ of Lemma \ref{upluslemma}.
We adapt the requirements to our case:
\begin{enumerate}
    \item $H_{i+1} \subset H_i$;
    \item $h_i < \min H_i$;
    \item $H_i$ is computable from $\TJ(c \oplus X, \b)$ for some $\beta < \alpha$;
    \item for each $s,t \in [H_i]^{! \alpha[h_i]}$ $c(\langle h_i \rangle \conc s) = c(\langle h_i \rangle \conc t)$.
\end{enumerate}

Initialize the construction by leaving $h_{-1}$ undefined and setting $H_{-1} = X$.
Suppose we have defined the sequences up to $h_{i-1}$ and $H_{i-1}$.
Let $h_i = \min H_{i-1}$ and consider the $c \oplus H_{i-1}$ computable coloring
\begin{align*}
    f_{h_i} \colon & [H_{i-1} \setminus \{h_i\}]^{! \alpha[h_i]} \to k \\
     & t \mapsto c(\langle h_i \rangle \conc t)
\end{align*}
Then since $\alpha[h_i] < \alpha$ by inductive hypothesis $\TJ(f_{h_i} \oplus H_{i-1}, \alpha[h_i] + 1)$ computes an infinite homogeneous set $H_i$ for $f_{h_i}$.
By the construction we know that $H_i$ is computable in $\TJ(c \oplus X, \b)$ and since $\alpha$ is indecomposable we also know that $\beta < \alpha$.
Therefore the sequences up to $i$ satisfy the requirements.

Let $Z = \{h_i : i \in \N \}$ and let $f \colon Z \to k$ be defined as $f(h_i) = f_{h_i} (t)$ for some $t \in [H_i]^{!\alpha[h_i]}$: by definition of $H_i$ the coloring $f$ is well defined and does not depend on the choice of $t \in [H_i]^{!\alpha[h_i]}$.
Notice that since $\alpha$ is indecomposable and each stage is computable in $\TJ(c \oplus X, \b)$ for $\beta < \alpha$, $\TJ(c \oplus X, \a)$ uniformly computes the sequences $(h_i)_{i \in \N}$ and $(H_i)_{i \in \N}$.
Since $f$ is an instance of $\RT^1_k$ which admits computable solutions, we get that $\TJ(c \oplus X,\a)$ (and so $\TJ(c \oplus X,\a+1)$) computes an infinite homogeneous set $H$ for $f$.
It is immediate to verify that $H$ is homogeneous for $c$ too.

\medskip

\textbf{Case 2:} Suppose that $\alpha$ is indecomposable and let $d \colon [Y]^{! 1 \uplus \alpha} \to k$.
The proof is completely analogous to the previous case so we only highlight the differences.
We construct a strictly increasing sequence $(h_i)_{i \in \N}$ of elements of $Y$ and a sequence $(H_i)_{i \in \N}$ of infinite subsets of $Y$.
The first 3 properties we require are exactly the same as in the previous case (replacing $X$ with $Y$ and $c$ with $d$), while property 4 is:
\begin{enumerate}[start=4]
    \item for each $s,t \in [H_i]^{! 1 \uplus \alpha[h_i]}$ $d(\langle h_i \rangle \conc s) = d(\langle h_i \rangle \conc t)$.
\end{enumerate}

The steps of the construction are the same, with the difference that the $d \oplus H_{i-1}$ computable coloring we consider is
\begin{align*}
    f_{h_i} \colon & [H_{i-1} \setminus \{h_i\}]^{! 1 \uplus \alpha[h_i]} \to k \\
     & t \mapsto d(\langle h_i \rangle \conc t)
\end{align*}
 
Let $Z = \{h_i : i \in \N \}$ and let $f \colon Z \to k$ be defined as $f(h_i) = f_{h_i} (t)$ for some $t \in [H_i]^{! 1 \uplus \alpha[h_i]}$, which is well defined as before.
Then we produce exactly as in the previous case an infinite homogeneous set $H$ for $f$ computable in $\TJ(d \oplus Y, \a+1)$, which is immediately verified to be homogeneous for $d$ too.

\medskip

\textbf{Case 3:} Suppose that $\alpha$ is decomposable, let $\a'<\a$ be such that $\alpha = \lead(\alpha) + \a'$ and let $c \colon [X]^{! \alpha} \to k$.
Since $\alpha$ is decomposable then $\alpha' > 0$ and by Definition \ref{system of fund seq} an $\alpha$-size set consists of an $\alpha'$-size set followed by a $\lead(\alpha)$-size set.
The strategy is similar to the previous cases, but it is slightly more delicate and requires additional care, so we provide the full details.
Fix a recursive enumeration $\{s_i: i \in \N\}$ of all the elements of $[X]^{!1 \uplus \alpha'}$ such that for all $n \in \N$, if $i$ is the least with $n \in s_i$ then each $(1 \uplus \alpha')$-size set of numbers strictly smaller than $n$ is enumerated before $s_i$. 
We aim to define a sequence $(h_i)_{i \in \N}$ of $(1 \uplus \alpha')$-size subsets of $X$ and a sequence $(H_i)_{i \in \N}$ of infinite subsets of $X$ with the following properties:
\begin{enumerate}
    \item $h_i$ comes before $h_{i+1}$ in the fixed enumeration $\{s_i : i \in \N \}$;
    \item $H_{i+1} \subset H_i$;
    \item $h_i < \min(H_i)$;
    \item $h_i \subset \bigcup_{j < i} h_j \cup H_{i-1}$;
    \item $H_i$ is computable from $\TJ(c \oplus X, \beta)$ for some $\beta < \lead(\alpha)$;
    \item for each $s,t \in [H_i]^{! \lead(\alpha)[\max h_i]}$ $c(h_i \conc s) = c(h_i \conc t)$.
\end{enumerate}

We make an observation about property 4.
At each stage $i$, if $s_\ell \subset \bigcup_{j < i} h_j \cup H_{i-1}$ we say that $s_\ell$ is an \textit{eligible set at stage $i$}.
Being eligible means that $s_\ell$ is a suitable candidate to be chosen as $h_i$ at stage $i$.
Notice that if $s_\ell \subset \bigcup_{j < i} h_j$ then for each $n \ge i$, $s_\ell$ will be eligible at stage $n$.

Initialize the construction by leaving $h_{-1}$ undefined and setting $H_{-1} = X$.
Suppose we have defined the sequences up to $h_{i-1}$ and $H_{i-1}$.
Let $h_i$ be the least eligible set at stage $i$ in the enumeration $\{s_i : i \in \N \}$ which is different from $h_j$ for each $j < i$.
Notice that $h_i$ was eligible also at stage $j$ for $j<i$ and so it must occur after each $h_j$ for $j<i$ in the enumeration $\{s_i : i \in \N \}$, as otherwise it would have been chosen at a previous stage.
Consider the $c \oplus H_{i-1}$ computable coloring
\begin{align*}
    f_{h_i} \colon & [H_{i-1} \setminus \{0, \ldots, \max h_i\}]^{! \lead(\alpha)[\max h_i]} \to k \\
     & t \mapsto c(h_i \conc t)
\end{align*}
Then by inductive hypothesis $\TJ(f_{h_i} \oplus H_{i-1}, \lead(\alpha)[\max h_i]+1)$ computes an infinite homogeneous set $H_i$ for $f_{h_i}$.
By the construction we know that $H_i$ is computable in $\TJ(c \oplus X, \b)$ and since $\lead(\alpha)$ is indecomposable we also know that $\beta < \lead(\alpha)$.
Therefore the sequences up to $i$ satisfy the requirements.

Let $Z = \bigcup_{i \in \N} h_i$.
We claim that for each $u \in [Z]^{! 1 \uplus \alpha'}$ there is $i$ such that $u = h_i$.
Let $u$ be a $(1 \uplus \a')$-size subset of $Z$ and say $u = s_\ell$ in the fixed enumeration.
Then for some $j$, $s_\ell \subset \bigcup_{i < j} h_i$: as we notice before, this means that $s_\ell$ will be eligible at any stage after stage $i$.
Since $s_\ell$ has only $\ell$ predecessors in the fixed enumeration $\{s_i : i \in \N \}$ of the $(1 \uplus \a')$-size sets, it will be chosen as one of the $h_i$ no later than stage $i+\ell$.

Let $f \colon [Z]^{! 1 \uplus \alpha'} \to k$ be defined as $f(h_i) = f_{h_i} (t)$ for $t \in [H_i]^{! \lead(\alpha)[\max h_i]}$: by definition of $H_i$ the coloring $f$ is well defined and does not depend on the choice of $t \in [H_i]^{! \lead(\alpha)[\max h_i]}$.
Notice that since $\lead(\alpha)$ is indecomposable and each stage is computable in $\TJ(c \oplus X, \beta)$ for some $\beta < \lead(\alpha)$, $\TJ(c \oplus X, \lead(\alpha))$ uniformly computes the sequences $(h_i)_{i \in \N}$ and $(H_i)_{i \in \N}$.
Therefore $f$ and $Z$ are computable from $\TJ(c \oplus X, \lead(\alpha))$ and consequently by inductive hypothesis $\TJ(\TJ(c \oplus X, \lead(\alpha)), \alpha'+1) = \TJ(c \oplus X, \a+1)$ computes an infinite homogeneous set $H$ for $f$.
It is immediate to verify that $H$ is homogeneous for $c$ too.

\medskip

\textbf{Case 4:} Suppose that $\alpha$ is decomposable, let $\a' < \a$ be such that $\alpha = \lead(\alpha) + \a'$ and let $d \colon [Y]^{! \alpha} \to k$.
Since $\alpha$ is decomposable then $\alpha' > 0$ and by Definition \ref{system of fund seq} an $\alpha$-size set consists of an $\alpha'$-size set followed by a $\lead(\alpha)$-size set.
The proof is completely analogous to the previous case so we only highlight the differences.
Consider the same recursive enumeration $\{s_i: i \in \N\}$ of the elements of $[Y]^{!1 \uplus \alpha'}$ as in the previous case.
We construct a sequence $(h_i)_{i \in \N}$ of $(1 \uplus \alpha')$-size subsets of $Y$ and a sequence $(H_i)_{i \in \N}$ of infinite subsets of $Y$.
The first 5 properties we require are the same as in the previous case (replacing $X$ with $Y$ and $c$ with $d$), while property 6 is:
\begin{enumerate}[start=6]
    \item for each $s,t \in [H_i]^{!1 \uplus \lead(\alpha)[\max h_i]}$ $d(h_i \conc s) = d(h_i \conc t)$.
\end{enumerate}

The steps of the construction are the same, with the difference that the $d \oplus H_{i-1}$ computable coloring we consider is
\begin{align*}
    f_{h_i} \colon & [H_{i-1} \setminus \{0, \ldots, \max h_i\}]^{! 1 \uplus \lead(\alpha)[\max h_i]} \to k \\
     & t \mapsto c(h_i \conc t)
\end{align*}

Let $Z = \bigcup_{i \in \N} h_i$.
The proof that for each $u \in [Z]^{! 1 \uplus \alpha'}$ there is $i$ such that $u = h_i$ is the same as in the previous case.
Let $f \colon [Z]^{! 1 \uplus \alpha'} \to k$ be defined as $f(h_i) = f_{h_i} (t)$ for $t \in [H_i]^{! \lead(\alpha)[\max h_i]}$, which is well defined as before.
Then we produce exactly as in the previous case an infinite homogeneous set $H$ for $f$ computable in $\TJ(d \oplus Y, \a+1)$, which is immediately verified to be homogeneous for $d$ too.
\end{proof}

The final corollary translates the results of Theorem \ref{upper general case} in the framework of reverse mathematics and conclude the proof of the upper bound and of the Main Theorem \ref{thm:main}.

\begin{corollary}[$\RCA$]
For each $\alpha < \Gamma_0$ and each $k \in \N$, $\Pi^0_{\lead(\alpha)} \CA \vdash \RT^{! \alpha}_k$.
\end{corollary}

\begin{proof}
By Theorem \ref{upper general case} we know that for each instance $c \colon [X]^{! \alpha} \to k$ of $\RT^{! \alpha}_k$, $\TJ(c \oplus X, \a+1)$ computes an infinite homogeneous set.
Hence $\Pi^0_{\alpha + 1} \CA \vdash \RT^{! \alpha}_k$.
By Lemma \ref{equivsystems} $\Pi^0_{\alpha + 1} \CA \leftrightarrow \Pi^0_{\lead(\alpha)} \CA$.
\end{proof}

\section{Conclusions and perspectives}


We have characterized the proof-theoretic strength of the family $\RT^{!\alpha}_k$ of Ramsey-like theorems for $\alpha$-size sets for $\alpha < \Gamma_0$. 
In particular we showed that, over $\RCA$, each principle $\RT^{!\alpha}_k$ is equivalent to the system $\Pi^0_{\lead(\alpha)} \CA$, providing a precise calibration of these combinatorial statements in reverse mathematics. 
This extends the analysis of $\RT^{!\omega}_2$ by Carlucci and Zdanowski \cite{carlucci2014} to all countable ordinals $\alpha<\Gamma_0$, establishing a clear hierarchy of logical strength corresponding to transfinite Turing jumps.

Our work highlights the interplay between largeness notions and the computational content of combinatorial principles. 
In particular, the $\alpha$-size sets capture the combinatorial strength needed to ensure homogeneous sets of high computational complexity since they allow us to code the $\alpha$-th Turing jump of a set.
Our reverse mathematics perspective complements previous computability-theoretic analyses of similar Ramsey-like theorems \cite{clote1984recursion}, providing formal equivalences within subsystems of second-order arithmetic.

From the perspective of proof strategies, our lower bound proof through reduction of a well-ordering principle extends the approach from \cite{carlucci_zdanowski_2012, carlucci2025}. 
It would be interesting to investigate into further extensions of such an approach to stronger Ramsey-like theorems and axiomatic systems. 

\begin{sloppypar}
The recent \cite{carlucci2024} analyzes the reverse mathematics and the computability-theoretic strength of the generalization to $\omega$-size sets of three weak Ramsey-like principles (the Free set, Thin set and Rainbow Ramsey theorems). 
Some results on the generalization of these principles to Nash-Williams barriers, including the $\alpha$-size sets are in \cite{carlucci_gjetaj_26}. A natural question is to inquire into the reverse mathematics of extensions of these principles to $\a$-size sets for $\a>\omega$.
\end{sloppypar}

\end{document}